\documentclass[11pt]{amsart}
\usepackage{etex}

\usepackage{amsmath}
\usepackage{mathrsfs, euscript}
\usepackage{bbm}
\usepackage{amssymb}
\usepackage{amscd}
\usepackage[all,cmtip]{xy}

\usepackage{enumerate}

\newtheorem{theorem}{Theorem}[section]
\newtheorem*{theorem*}{Theorem}

\newtheorem{lemma}[theorem]{Lemma}
\newtheorem*{lemma*}{Lemma}
\newtheorem{corollary}[theorem]{Corollary}
\newtheorem{proposition}[theorem]{Proposition}

\newtheorem{remark}[theorem]{Remark}
\newtheorem{definition}[theorem]{Definition}

\newcommand{\hh}{\mathcal{H}}
\newcommand{\LL}{\mathcal{L}}
\newcommand{\F}{\mathcal{F}}

\newcommand{\A}{\mathcal{A}}

\newcommand{\W}{\mathcal{W}}

\newcommand{\C}{\mathbb{C}}

\newcommand{\mrx}{\mathcal{X}^{\circ}}
\newcommand{\X}{\mathcal{X}}

\newcommand{\talpha}{\tilde{\alpha}}

\usepackage[all]{xy}

\begin{document}
\unitlength=1mm
\special{em:linewidth 0.4pt}
\linethickness{0.4pt}
\title[Free-Boolean]{ Free-Boolean independence for  pairs of algebras}
\author{Weihua Liu}
\maketitle
\begin{abstract}

 We  construct  pairs of algebras with mixed independence relations by using truncations of reduced free products of algebras. 
  For example,   we construct free-Boolean pairs of algebras and free-monotone pairs of algebras. 
We also introduce free-Boolean cumulants and show that free-Boolean independence is equivalent to the vanishing of mixed cumulants.
\end{abstract}

\section{Introduction}

In noncommutative probability,  independence relations between random variables  provide specific rules for calculations of  all mixed moments of those random variables.  
Most  independence relations are realized by certain universal products of subalgebras.   
Among those independence relations which arise from universal products,   the strongest one is the classical independence and the tensor product is the corresponding universal product. 
Here,  strongest means that the universal product satisfies the greatest number of axioms, for instance,  commutativity of subalgebras,  commutativity of the universal product,  unital,  associativity.  
By decreasing the number of axioms,  we obtain additional universal products and independence relations.
 It is shown in \cite{Sp2} that  there are  exactly three commutative and associative universal products:  tensor,  free and Boolean.  
  Here,   commutative means that the product does not depend on the order of the subalgebras.  
   Once we remove the condition of commutativity,  we abtain two additional universal products, namely the  monotone and the anti-monotone products \cite{Mu}.  
   If we drop the associativity requirement, then there are infinitely many universal products.  
   For example, the interpolated free product which is studied by  Franz and Lenczewski \cite{FL} and the  hierarchy of monotone products   which is studied by Lenczewski and Sa\l apata \cite{LS}.
    Another important one is the $c-$free product  which plays an important role in the subordination property of free convolution \cite{Lenczewski}.  
    Except for the tensor product, all  universal products can be constructed via certain truncations of the reduced product of algebraic probability spaces. We review these constructions in Section 2.    

Recently,  Voiculescu\cite{Voi1} introduced  bi-free probability for pairs of faces thus generalizing free probability.   
The idea  is to study, at the same time, left  and right operators on reduced free products  of vector spaces with specified vectors. 
The combinatorial aspects of bi-free product are studied in \cite{CNS,MN}, where bi-free cumulants are introduced and bi-freeness is characterized by the vanishing of mixed cumulants.
More recently,  more independence relations for pairs of algebras are studied. For example
in \cite{GHS, GS, GS1},   conditionally bi-free independence, bi-Boolean independence, bi-monotone independence are introduced and studied.

One purpose of this paper is to  introduce pairs of families of  random variables  that display different independence relations. For instance, the left face of the random variables  are freely independent and the right face of the random variables  are Boolean independent.  
The construction of mixed independent pairs of algebras is obtained from Voiculescu's  bifree independence by  truncations of the reduced free product of probability spaces.  
We can also define bi-Boolean and bi-monotone independence with our construction.  
However, our definition of bi-Boolean is trivial and is different from the bi-Boolean case in \cite{GS}.  
On the other hand,  the bi-monotone independence relation in our sense is the same as the type I bi-monotone independence in \cite{GHS}.

Among the family of pairs of algebras with different independence relations,  the free-Boolean one is quite interesting because that the free-Boolean product is commutative and associative.  
Therefore, we can construct free-Boolean exchangeable sequences  and we expect a vanishing-cumulant theorem \cite{Lehner}. 
As in the combinatorial theory of free probability, we introduce an associated family of partitions which we call interval-noncrossing partitions.  
We show that the family of interval-noncrossing partitions has a natural lattice structure. 
Then, with the help of the M\"obius inversion functions,  we  define free-Boolean cumulants and  prove a vanishing-cumulant condition for free-Boolean independence.   This allows us to obtain the central limit law  for free-Boolean independence.

The paper is organized as follows: 
 In Section 2, we review the construction of the bi-free product. 
  We show all associative universal products, with the exception of tensor product, can be obtained by restricting   to certain subspaces of the reduced free product  of vector spaces with specified vectors. Therefore, by compressing left faces and right faces of pairs of random variables to certain subspaces will give mixed independent pairs of random variables.
 In Section 3,  we introduce the precise notions for  mixed independence relations 
  In Section 5,  we introduce a notion of interval-noncrossing partitions which will be used  to study free-Boolean pairs of algebras. We also study the lattice structure of sets of interval-noncrossing partitions.
  In Section 4,  we give an equivalent definition for free-Boolean independence relation via conditions of mixed moments.
  In Section 6,  we study  the M\"obius inversion functions on the lattice of interval noncrossing partitions.  Free-Boolean cumulants and combinatorially free-Boolean independence are introduced.
  In Section 7,  we show that combinatorially free-Boolean independence is equivalent to the algebraically free-Boolean independence.
  In Section 8, as an application of the main theorem in Section 6,  we will study the free-Boolean central limit laws.

\section{Preliminaries}

We begin by recalling  the left regular representations and the right regular representations on reduced free product spaces with specified vectors.   The reader is  referred to \cite{Voi1} for further information.

\begin{definition}\normalfont  A \emph{vector space with a specified vector} is a  triple $(\X,\mrx, \xi)$ where $\mathcal{X}$ is a vector space, $\X$ is a  subspace of $\mathcal{X}$ codimension $1$ and $\xi\in\mathcal{X}\setminus\mrx$.
\end{definition}

Given a  vector space with a specified vector $(\mathcal{X},\mrx, \xi)$. Observe that $\X=\C\xi\oplus\mrx $,  there exists a unique linear functional $\phi$ on $\mathcal{X}$  such that $\phi(\xi)=1$  and $\ker(\phi)=\mrx$.  We  denote by $\mathcal{L}(\mathcal{X})$ the algebra of linear operators on $\mathcal{X}$ and we define a  linear functional $\phi_\xi:\LL(\X)\rightarrow \C$ such that $\phi_\xi(T)=\phi(T\xi), T\in \LL(X)$.  

Given a family of vector spaces with specified vectors  $(\X_i, \mrx_i,\xi_i)_{i\in I}$ , the  reduced free product  space $(\mathcal{X}, \mrx,\xi)=\underset{i\in I}{*}(\X_i, \mrx_i,\xi_i)$ is given by $\X=\C\xi\oplus\mrx$ where
 $$\mrx=\bigoplus\limits_{n\geq 1}\Big(\bigoplus\limits_{i_1\neq i_2\neq\cdots \neq i_n} \mathring{\mathcal{X}_{i_1}}\otimes\cdots \otimes\mrx_{i_n}\Big). $$

For every $i\in I$, we let
$$\mathcal{X}(\ell,i)=\mathbb{C}\xi\oplus\bigoplus\limits_{n\geq 1}(\bigoplus\limits_{i_1\neq i_2\neq\cdots \neq i_n, i_1\neq i} \mathring{\mathcal{X}_{i_1}}\otimes\cdots \otimes\mrx_{i_n})$$
and
$$\mathcal{X}(r,i)=\mathbb{C}\xi\oplus\bigoplus\limits_{n\geq 1}(\bigoplus\limits_{i_1\neq i_2\neq\cdots \neq i_n, i_n\neq i} \mathring{\mathcal{X}_{i_1}}\otimes\cdots \otimes\mrx_{i_n}).$$
As was shown in  \cite{Voi1}, there are natural  linear isomorphisms: $V_i:\mathcal{X}_i\otimes \mathcal{X}(\ell,i) \rightarrow \mathcal{X}$ and $W_i: \mathcal{X}(r,i)\otimes\mathcal{X}_i \rightarrow \mathcal{X}$.
Therefore,  for each $i\in I$, the algebra $\mathcal{L}(\mathcal{X}_i)$ has a left representation $\lambda_i$ and a right representation $\rho_i$, on $\X$, which are given by
$$\lambda_i(T)=V_i(T\otimes I_{\mathcal{X}(\ell,i)})V_i^{-1}$$
and 
$$\rho_i(T)=W_i(I_{\mathcal{X}(r,i)}\otimes T )W_i^{-1}$$ 
for every $T\in \mathcal{L}(\mathcal{X}_i)$, where $I_{\mathcal{X}(r,i)}$ and $I_{\mathcal{X}(\ell,i)}$ are the identity operators on ${\mathcal{X}(r,i)}$ and ${\mathcal{X}(\ell,i)}$ respectively.

In noncommutative probability, except the tensor product, all the other four associative universal products  have connections to  reduced free products of vector spaces with specified vectors.  A \emph{noncommutative probability space} is a pair $(\A,\phi)$ where $\A$ is an algebra and $\phi$ is a linear functional on $\A$ such that $\phi(1_{\A})=1$. 
The independence relations associated with the universal products are defined as follows:
\begin{definition}\normalfont   Let $I$ be an index set and  $(\A,\phi)$ be a noncommutative probability space.\\
 A family of unital subalgebras $\{\A_i\}_{i\in I}$ of $\A$  said to be \emph{freely independent}  if
$$\phi(x_1\cdots x_n)=0,$$
whenever $x_k\in\A_{i_k}$, $i_k\neq i_{k+1}$ and $\phi(x_k)=0$ for all k. \\
A family of (not necessarily unital) subalgebras $\{\A_i|i\in I \}$ of $\A$ is said to be \emph{Boolean independent} if
$$\phi(x_1x_2\cdots x_n)= \phi(x_1)\phi(x_2)\cdots \phi(x_n),$$
whenever $x_k\in\A_{i_k}$, $i_k\neq i_{k+1}$ for all $k$. \\
Suppose that $I$ is totally ordered. The family of subalgebras $\{\A_i\}_{i\in I}$  is said to be monotone independent if 
$$\phi(x_1\cdots x_{k-1}x_kx_{k+1}\cdots x_n)=\phi(x_k)\phi(x_1\cdots x_{k-1}x_{k+1}\cdots x_n)$$
whenever $x_j\in\A_{i_j}$,  $i_j\neq i_{j+1}$ and $i_{k-1}<i_k>i_{k+1}$. \\
By reversing the order of $I$, we will have the so called anti-monotone independence relation.\\
A set of random variables $\{x_i\in\A|i\in I\}$ is said to be freely(respectively Boolean, monotone) independent if the family of unital(respectively non-unital) subalgebras $\A_i$, which are generated by $x_i$'s respectively, are freely (respectively Boolean, monotone) independent.
\end{definition} 
\begin{remark}\normalfont In the preceding definition, \lq\lq non-unital\rq\rq means that the subalgebra does not contain the unit of  $\A$. 
\end{remark}
  
Now, we briefly exhibit some relations between independence relations and truncated reduced products: 

{\bf Free independence:} Suppose that $(\mathcal{X}, \mrx,\xi)=\underset{i\in I}{*}(\X_i, \mrx_i,\xi_i)$ , then the family $\{\lambda_i(\mathcal{L}(\X_i))\}_{i\in I}$  of subalgebras of $\mathcal{L}(\X)$ is  freely independent in $(\mathcal{L}(\X),\phi_{\xi})$.  Similarly,  $\{\rho_i(\mathcal{L}(\X_i))\}_{i\in I}$ is a family of freely independent subalgebras in $(\mathcal{L}(\X),\phi_{\xi})$.

{\bf Boolean independence:}   Given a family of vector spaces with specified vectors  $(\mathcal{X}_i, \mrx_i,\xi_i)_{i\in I}$,  the Boolean product space is defined as $(\mathcal{X}_\uplus, \mrx_\uplus,\xi)=\underset{i\in I}{\uplus}(\mathcal{X}_i, \mrx_i,\xi_i)$ where $\X_\uplus=\mathbb{C}\xi\oplus\mrx_\uplus$ and
$$ \mrx_\uplus= \bigoplus\limits_{i\in I}\mrx_i.$$ 
For each $i\in I$, there is a natural projection $P_{\uplus,i}:\X_{\uplus}\rightarrow \C\xi\oplus\mrx_i$ such that $P_{\uplus,i}(x)=x$ if $x\in\C\xi\oplus\mrx_i$ and $P_{\uplus,i}(x)=0$ if $x\in  \bigoplus\limits_{j\neq i }\mathring{\mathcal{X}_j}$.  Since  $\C\xi\oplus\mrx_i$ can be identified with $\X_i$, we can define a  linear map $\alpha_i: \LL(\mathcal{X}_i) \rightarrow B(\mathcal{X}_{\uplus})$  as for follows:  $$\alpha_i(T)=P_{\uplus,i}TP_{\uplus,i},$$
for  every $T\in \LL(\X_i)$.  It is obvious that $\alpha_i$ is a algebra homomorphism but not  unital.

\begin{proposition}\normalfont  Given a family of vector spaces with specified vectors  $(\mathcal{X}_i, \mrx_i,\xi_i)_{i\in I}$, let  $(\mathcal{X}_\uplus, \mrx_\uplus,\xi)$ be the Boolean product of them. Then,  the family $\{\alpha_i(\LL(\X_i))\}$ is Boolean independent with respect to $\phi_\xi$, where $\phi_\xi$ is the linear functional associated with $(\mathcal{X}_\uplus, \mrx_\uplus,\xi)$.
\end{proposition}
\begin{proof}
Let $\{x_1 ,\cdots,x_n\}$ be a sequence of linear maps on $\X_\uplus$ such that $x_k\in\alpha_{i_k}(\LL(\X_{i_k}))$ and $i_1\neq \cdots \neq i_n$.  
Notice that $x_n\xi-\phi_\xi(x_n)\xi\in\mrx_{i_n}$,  then $x_{n-1}(x_n\xi-\phi_\xi(x_n)\xi)=0$. We have
$$x_{n-1}x_n\xi=x_{n-1}(\phi(x_n)\xi +(x_n\xi-\phi_\xi(x_n)\xi)=\phi_\xi(x_n)x_{n-1}\xi.$$
It follows by induction that 
$$\phi(x_1\cdots x_n)=\phi(x_1\cdots x_{n-1})\phi(x_n)=\cdots=\phi(x_1)\phi(x_2)\cdots \phi(x_n).$$
\end{proof}

We see  that the Boolean product space $\X_\uplus$ is a subspace of the reduced product $\X$. Actually, we have $$\X=\X_\uplus\oplus\bigoplus\limits_{n\geq 2}\Big(\bigoplus\limits_{i_1\neq i_2\neq\cdots \neq i_n} \mathring{\mathcal{X}_{i_1}}\otimes\cdots \otimes\mathring{\mathcal{X}_{i_n}}\Big). $$ 
Thus, there is a projection $P_\uplus: \X\rightarrow \X_\uplus$  such that $P_\uplus(x)=x$ whenever $x\in \X_\uplus$ and $P_\uplus(x)=0$ if $x\in\oplus\bigoplus\limits_{n\geq 2}\Big(\bigoplus\limits_{i_1\neq i_2\neq\cdots \neq i_n} \mathring{\mathcal{X}_{i_1}}\otimes\cdots \otimes\mathring{\mathcal{X}_{i_n}}\Big)$.  Therefore, $P_{\uplus,i}$ can be defined on $\X$ such that $P_{\uplus,i}=P_{\uplus,i}P_{\uplus}$ and we have 
  $$ \alpha_i(T)= P_{\uplus,i}\lambda_i(T)P_{\uplus,i}=P_{\uplus,i}\rho_i(T)P_{\uplus,i}$$
  for all $i$, $T\in \LL(\X_i)$.   

\begin{corollary}\normalfont   Given a family of vector spaces with specified vectors  $(\mathcal{X}_i, \mrx_i,\xi_i)_{i\in I}$ .
Then,  the family $\{P_{\uplus,i}\lambda_i(\LL(\X_i))P_{\uplus,i}\}_{i\in I}$=$\{P_{\uplus,i} \rho_i(\LL(\X_i))P_{\uplus,i}\}_{i\in I}$ are Boolean independent with respect to $\phi_\xi$, where $\phi_\xi$ is the linear functional associated with $(\X,\mrx,\xi)$. 
\end{corollary}

{\bf Monotone independence}
As in the case of  Boolean independence, we can use certain  projections on $(\mathcal{X}, \mrx,\xi)=\underset{i\in I}{*}(\X_i, \mrx_i,\xi_i)$  to construct monotone independent families of subalgebras. Suppose that $I$ is an ordered set,  the monotone product space $(\mathcal{X}_{\rhd}, \mrx,\xi)=\underset{i\in I}{\rhd}(\mathcal{X}_i, \mrx_i,\xi_i)$ is defined by
$$ \mathcal{X}_{\rhd}=\mathbb{C}\xi\oplus\bigoplus\limits_{n\geq 1}\Big(\bigoplus\limits_{i_1> i_2>\cdots > i_n} \mathring{\mathcal{X}_{i_1}}\otimes\cdots \otimes\mathring{\mathcal{X}_{i_n}}\Big) .$$
Let 
$$\X(\rhd,i)=\C\xi\oplus\bigoplus\limits_{n\geq 1}\Big(\bigoplus\limits_{i=i_1> i_2>\cdots > i_n} \mathring{\mathcal{X}_{i_1}}\otimes\cdots \otimes\mathring{\mathcal{X}_{i_n}}\Big)$$
and 
$$\X(\rhd,\ell,i)=\C\xi\oplus\bigoplus\limits_{n\geq 1}\Big(\bigoplus\limits_{i> i_1>\cdots > i_n} \mathring{\mathcal{X}_{i_1}}\otimes\cdots \otimes\mathring{\mathcal{X}_{i_n}}\Big).$$

We view the above spaces as subspaces of the reduced free product of $(\mathcal{X}_i, \mrx_i,\xi_i)_{i\in I}$.  Moreover, for each $i$, $\X(\rhd,\ell,i)$ is a subspace of $\X(\ell,i)$. There is a natural isomorphism from  $\X_i\otimes \X(\rhd,\ell,i)$ to  $\X(\rhd,i)$ such that 
$$ \xi_i\otimes \xi\rightarrow \xi$$
$$ \mrx_i\otimes \xi\rightarrow \mrx_i$$
$$ \xi_i\otimes \mathring{\mathcal{X}_{i_1}}\otimes\cdots \otimes\mathring{\mathcal{X}_{i_n}}\rightarrow \mathring{\mathcal{X}_{i_1}}\otimes\cdots \otimes\mathring{\mathcal{X}_{i_n}}$$
$$ \mrx_i\otimes (\mathring{\mathcal{X}_{i_1}}\otimes\cdots \otimes\mathring{\mathcal{X}_{i_n}})\rightarrow \mrx_i\otimes \mathring{\mathcal{X}_{i_1}}\otimes\cdots \otimes\mathring{\mathcal{X}_{i_n}}.$$
Actually, the above isomorphism is the restriction of   $V_i$ to $\X_i\otimes \X(\rhd,\ell,i)$.  For each $i$,  there is a projection
$P_{\rhd,i}:\X\rightarrow \X(\rhd,i)$ 
such that $P_{\rhd,i}(x)=x$ if $x\in\X(\rhd,i)$ and $P_{\rhd,i}(x)=0$ on the other summands .  Therefore, for each $i$,   $\mathcal{L}(\X_i)$ has a representation $\lambda'_i$ on $\X$ such that 
$$\lambda'_i(T)=P_{\rhd,i}V_i(T\otimes I_{\X(\rhd,\ell,i)})V_i^{-1}P_{\rhd,i}$$
for all $T\in \LL(\X_i)$.
\begin{proposition} \label{construction of monotone }\normalfont  Given a family of vector spaces with specified vectors  $(\mathcal{X}_i, \mrx_i,\xi_i)_{i\in I}$, where $(I,>)$ is a totally ordered index set, let  $(\mathcal{X}, \mrx,\xi)=\underset{i\in I}{\rhd}(\mathcal{X}_i, \mrx_i,\xi_i)$ be their monotone product. Then,  the family $\{\lambda'_i(\LL(\X_i))\}$ is monotone independent with respect to $\phi_\xi$, where $\phi_\xi$ is the linear functional associated with $(\X,\mrx,\xi)$.
\end{proposition}
\begin{proof}
Let $\{x_1 ,\cdots,x_n\}$ be a sequence of linear maps in $\mathcal{\X}$ such that $x_j\in\lambda'_{i_j}(\LL(\X_{i_j}))$ and $i_1\neq \cdots \neq i_n$.  Assume that there is a $k$ such that $i_k> i_{k-1}$ and $i_k>i_{k+1}$. \\
Note that $x_{k+1}\cdots x_n\xi\in \X(\rhd,i_{k+1})\subseteq \X(\rhd,\ell ,i_{k})$.  Suppose that  $x_k=P_{\rhd,i_k}V_{i_k}(T\otimes id_{ \X(\rhd,\ell ,i_{k})}) V_{i_k}^{-1}P_{\rhd,i_k}$ for some $T\in \LL(\X_{i_k})$. Then 
$$
\begin{array}{rcl}
x_k(x_{k+1}\cdots x_n\xi)&=&P_{\rhd,i_k}V_{i_k}(T\otimes id_{ \X(\rhd,\ell,i_{k})}) V_{i_k}^{-1}P_{\rhd,i_k}(x_{k+1}\cdots x_n\xi)\\
&=&P_{\rhd,i_k}V_{i_k}(T\otimes id_{ \X(\rhd,\ell ,i_{k})})(\xi_{i_k}\otimes(x_{k+1}\cdots x_n\xi))\\
&=&P_{\rhd,i_k}V_{i_k}\big(\phi_{i_k}(T)\xi_{i_k}\otimes(x_{k+1}\cdots x_n\xi)+(T\xi_{i_k}-\phi_{i_k}(T)\xi_{i_k})\otimes(x_{k+1}\cdots x_n\xi))\\
&=&\phi_{i_k}(T)x_{k+1}\cdots x_n\xi+(T\xi_{i_k}-\phi_{i_k}(T)\xi_{i_k})\otimes(x_{k+1}\cdots x_n\xi)\\
\end{array}
$$
where $\phi_{i_k}$ is the linear functional associated with $(\X_{i_k},\mrx_{i_k},\xi_{i_k})$.
Therefore, 
$$
\begin{array}{rcl}
x_{k-1}x_k(x_{k+1}\cdots x_n\xi)&=&x_{k-1}(\phi_{i_k}(T)x_{k+1}\cdots x_n\xi+(T\xi_{i_k}-\phi_{i_k}(T)\xi_{i_k})\otimes(x_{k+1}\cdots x_n\xi))\\&=&\phi_{i_k}(T)x_{k-1}(x_{k+1}\cdots x_n\xi)+x_{k-1}(T\xi_{i_k}-\phi_{i_k}(T)\xi_{i_k})\otimes(x_{k+1}\cdots x_n\xi))\\
\end{array}
$$
Since $(T\xi_{i_k}-\phi_{i_k}(T)\xi_{i_k})\otimes(x_{k+1}\cdots x_n\xi))\in \bigoplus\limits_{n\geq 1}(\bigoplus\limits_{i_{k}=l_1>\cdots>l_n}\mrx_{l_1}\otimes\cdots\otimes\mrx_{l_n})$ of which no component is contained in the range of $P_{\rhd,i_{k-1}}$,  we have 
$$x_{k-1}x_k(x_{k+1}\cdots x_n\xi)=\phi_{i_k}(T)x_{k-1}(x_{k+1}\cdots x_n\xi).$$
Notice that $\phi_{\xi}(x_k)=\phi_{i_k}(T).$ We have 
$$\phi_{\xi}(x_1\cdots x_{k-1}x_kx_{k+1}\cdots x_n)=\phi_{\xi}(x_k)\phi(x_1\cdots x_{k-1}x_{k+1}\cdots x_n).$$
\end{proof}

A monotone independent  sequence of algebras can also be constructed by using the right decomposition of the  reduced products:
Given a family of vector spaces with specified  vectors  $(\mathcal{X}_i, \mrx,\xi_i)_{i\in I}$, the anti-monotone product space $(\mathcal{X}_{\lhd}, \mrx,\xi)=\underset{i\in I}{\lhd}(\mathcal{X}_i, \mrx,\xi_i)$ is given by 
$$ \mathcal{X}_{\lhd}=\mathbb{C}\xi\oplus\bigoplus\limits_{n\geq 1}\Big(\bigoplus\limits_{i_1< i_2<\cdots < i_n} \mathring{\mathcal{X}_{i_1}}\otimes\cdots \otimes\mathring{\mathcal{X}_{i_n}}\Big) .$$
In analogy to the monotone case, we get
$$\X(\lhd,i)=\C\xi\oplus\bigoplus\limits_{n\geq 1}\Big(\bigoplus\limits_{i_1< i_2<\cdots < i_n=i} \mathring{\mathcal{X}_{i_1}}\otimes\cdots \otimes\mathring{\mathcal{X}_{i_n}}\Big)$$
and 
$$\X(\lhd,r,i)=\C\xi\oplus\bigoplus\limits_{n\geq 1}\Big(\bigoplus\limits_{i_1<\cdots <i_n<i} \mathring{\mathcal{X}_{i_1}}\otimes\cdots \otimes\mathring{\mathcal{X}_{i_n}}\Big).$$
All the above spaces are viewed as  subspaces of  the reduced free product of  $(\mathcal{X}_i, \mrx_i,\xi_i)_{i\in I}$, and  $\X(\lhd,r,i)$ is a subspace of $\X(r,i)$. Therefore,  the restriction of $W_i$ to $ \X(\rhd,\ell,i)\otimes \X_i$ is an isomorphism onto $ \X(\rhd,i)$.
Let $P_{\lhd,i}$ be the projection from $\X$ to the subspace $\X(\lhd,i).$ Then, as  in the  monotone case, $\{P_{\lhd,i}\rho_i(\LL(\X_i))P_{\lhd,i}\}_{i\in I}$ is a monotone independent sequence of  algebras with respect to $\phi_{\xi}$.  On the other hand, for anti-monotone independence, we have the following result.

\begin{proposition} \normalfont  Given a family of vector spaces with specified vectors  $(\mathcal{X}_i, \mrx_i,\xi_i)_{i\in I}$, where $(I,>)$ is totally ordered set, let  $(\X,\mrx,\xi)$ be their reduced free product.   Then,  the families $\{P_{\rhd,i}\rho_i(\LL(\X_i))P_{\rhd,i}\}$  and the families $\{P_{\lhd,i}\lambda_i(\LL(\X_i))P_{\lhd,i}\}$ are anti-monotone independent families with respect to $\phi_\xi$, where $\phi_\xi$ is the linear functional associated with $(\X,\mrx,\xi)$.
\end{proposition}

Notice that,  all the above constructions rely on  truncations of reduced free products and left-right regular representations of certain algebras . The hierarchy of freeness  of Franz and Lenczewski\cite{FL} and the  hierarchy of monotone  of Lenczewski and Sa\l apata\cite{LS} can be obtained by using some other projections in place of  $P_{*,i}$.

\section{ Free-Boolean independence}

In this section,  we introduce the notion of mixed independent of pairs of faces via reduced free product spaces. 
Since Boolean and monotone products are not unital,  we must modify Voiculescu's definition of pairs of faces \cite{Voi1}. In the following definition,  we do not require subalgebras to be unital.

\begin{definition} \normalfont 
A \emph{pair of faces} in a noncommutative probability space $(\A,\phi)$ is an ordered pair $((B,\beta),(C,\gamma))$ where  $B,C$ are (not necessarily unit) algebras and $\beta:B\rightarrow \A$, $\gamma:C\rightarrow \A$ are algebra homomorphisms(not necessarily unit)  even if $B$ or $C$ has unite .  If $B,C$ are subalgebras of $\A$  and $\beta,\gamma$ are the  inclusion maps , then the pair will be simply  denoted by $(B,C)$.
\end{definition}

We can also use the original definition of pairs of faces here,  but then we just need turn to consider the unitalizations of $B,C$. In the following context, we will denote by $\bigstar$ the non-unital universal free product which can be constructed as follows:
Given a family $\{A_i\}_{i\in I}$ of algebras , we denote by $\widetilde{A_i}$ be the unitalization of $A_i$. Let $\underset{i\in I}{*}\widetilde{A_i}$ be the unital universal free product. Then, for each $i$, there is an inclusion $\iota_i:\widetilde{A_i}\rightarrow \underset{i\in I}{*}\widetilde{A_i}$.  The algebra generated by $\{\iota_i(A_i)\}_{i\in I}$ is the non-unital universal free product of  $\{A_i\}_{i\in I}$, we denote the algebra by $\underset{i\in I}{\bigstar}A_i$.

\begin{definition}\normalfont  If $\Gamma=\{(B_i,\beta_i),(C_i,\gamma_i)\}_{i\in I}$ is a family of pairs in $(\A,\phi)$, then its \emph{joint distribution} is the  functional $\mu_\Gamma: \underset{i\in I}{\bigstar} (B_{i}\bigstar C_i)$ defined by $\mu_\Gamma=\phi\circ \alpha$, where $\alpha:\underset{i\in I}{\bigstar} (B_{i}\bigstar C_i)\rightarrow \A$ is the unique algebra homomorphism such that $\alpha|_{B_i}=\beta_i$, $\alpha|_{C_i}=\beta_{i}$. 
\end{definition}

\begin{definition}\normalfont   A two faced family of noncommutative random variables in a noncommutative probability space  is an ordered pairs $a=\{(b_i)_{i\in I},(c_i)_{j\in J}\}$ in  $(\A, \phi) $ i.e. the $b_i$ and $c_j$ are elements of $\A$.  The distribution $\mu_a$ of $a$ is the linear functional 
$$\mu_a:\C\langle X_i, Y_j|i\in I, j\in J\rangle \rightarrow \C$$ 
such that $\mu_a=\phi\circ \alpha $ where $\alpha:\C\langle X_i, Y_j|i\in I, j\in J\rangle \rightarrow \A$ is the unique algebra homomorphism  such that $\alpha(X_i)=b_i$ and $\alpha(Y_j)=c_j$.
\end{definition}

By modifying Voiculescu's definition of bi-freeness, we have the following definition for free-Boolean independent pairs of algebras
\begin{definition}\label{free-Boolean} \normalfont  Let $\Gamma=\{(B_i,\beta_i),(C_i,\gamma_i)\}_{i\in I}$ be a family of  pairs of faces in $(\A,\phi)$, where $I$ is an index set. 
Suppose that  there is a family of vector spaces with specified vectors $ (\mathcal{X}_i, \mrx_i,\xi_i)_{i\in I}$ and  homomorphisms $\ell_i:B_i\rightarrow \mathcal{L}(\X_i)$, $r_i:C_i\rightarrow \LL(\X_i)$.  
Let  $( \X, \mrx,\xi)$ be the reduced free product of $ (\mathcal{X}_i, \mrx_i,\xi_i)_{i\in I}$ and let $\phi_\xi$ be the  associated linear functional on $\LL(X)$.  Suppose that  there are projections $P_i$ and $Q_i$ in $\LL(\X)$ such that the joint distribution of $ \{(B_i,P_i\lambda_i(\ell_i(\cdot))P_i),(C_i,Q_i\rho_i(r_i(\cdot))Q_i)\}_{i\in I}$ in $(\LL(X),\phi_{\xi})$, is equal to the joint distribution of $\{(B_i,\beta_i),(C_i,\gamma_i)\}_{i\in I}$. Then,
\begin{itemize}
\item $\Gamma$ is said to be\emph{ bi-freely independent} if $P_i=Q_i=Id_{\LL(X)}$ for all $i\in I$.
\item $\Gamma$ is said to be \emph{free-Boolean independent} if $P_i=Id_{\LL(X)} $ and $Q_i=P_{\uplus,i}$ for all $i\in I$.
\end{itemize}
\end{definition}

Notice that the independence relations of pairs depend on the choice of $P_i$ and $Q_i$.   If we replace $P_i$, $Q_i$ by other projections on $\X$, we would have
\begin{itemize}
\item $\Gamma$ is \emph{free-monotone}  independent if $P_i=Id_{\LL(X)} $ and $Q_i=P_{\lhd,i}$ for all $i\in I$.
\item $\Gamma$ is \emph{bi-monotone independent} if $P_i=P_{\rhd,i} $ and $Q_i=P_{\lhd,i}$ for all $i\in I$.
\item $\Gamma$ is \emph{monotone-anti-monotone} independent if $P_i=Q_i=P_{\rhd,i}$ for all $i\in I$.
\end{itemize}

One can ,of course, construct many other kinds independence for pairs of algebras by choosing proper projections $P_i$ and $Q_i$.  The reason that we call the family an  A-B independent family is that  the first face satisfies A-independence relation and the second face satisfies B-independence relation. A, B can be free, Boolean monotone, anti-monotone, hierarchy of freeness, etc.  One can also generalize the idea of a pair of faces to an $n$-tuple of faces. The special case of  the bifree-Boolean independence relation will be studied in a forthcoming paper.

\begin{remark}\normalfont  In  the preceding definition,  we can also define Boolean-Boolean independent families of algebras . However,  the Boolean-Boolean families of pairs of algebras in this sense are just Boolean independent sequences since there is no distinction between the left face and the right face. This makes the Boolean-Boolean independence trivial.   Gu and Skoufranis defined a nontrivial notion of  bi-Boolean independent families starting from the idea of two-state freeness. 
\end{remark}

It is  routine to show that the preceding definition does not depend on a particular choice of $l_k$, $r_k$ and $ (\mathcal{X}_i, \mrx,\xi_i)_{i\in I}$.  Given an A-B independent family of pairs of faces $\Gamma=\{(B_i,\beta_i),(C_i,\gamma_i)\}_{i\in I}$, one can see that the joint distribution $\mu_{\Gamma}$ is uniquely determined in Section 2.1 and 2.9 of \cite{Voi1}.  

Now, we turn to study free-Boolean pairs in details. The following is a definition of monotone independence for pairs of algebras.

\begin{definition}\normalfont 
Let $B, C$ be two  subalgebras of a probability space $(A,\phi)$, we say $B$ is \emph{monotone} to $C$ if  $$\phi(x_1\cdots x_{k-1}x_kx_{k+1}\cdots x_n)=\phi(x_k)\phi(x_1\cdots x_{k-1}x_{k+1}\cdots x_n)$$
whenever $x_{k-1},x_{k+1}\in C$ and $x_k\in B$.
\end{definition}

\begin{lemma}\label{FBtoM} \normalfont  Given two vector spaces with specified vectors $(\X_i,\mrx_i,\xi_i)_{i=1,2}$, let $(\X, \mrx,\xi)$ be their reduced free product, and $\lambda_i$ be the left regular representation of $\LL(\X_i)$ on $\X$, $i=1,2$. Then  the algebra $\lambda_1(\LL(\X_1))$ is monotone to $ P_{\uplus,1}\lambda_2(\LL(\X_2))P_{\uplus,1}$ with respect to $\phi_\xi$,  where $\phi_\xi$ is the linear functional associated with $(\mathcal{X}, \mrx,\xi)$.
\end{lemma}
\begin{proof}
Note that,  for $T\in \LL(\X_1)$,   $\lambda_1(T)= V_1(T\otimes I_{\X(\ell,1)})V_1^{-1}$, where $V_1$ is the isomorphism defined in Section 2.  Let 
$$\bar{\X}=\C\xi\oplus \mrx_1\oplus \mrx_2\oplus \mrx_1\otimes\mrx_2. $$
Then $V_1^{-1}\bar{\X}=\X_1\otimes (\C\xi\oplus \mrx_2)$ and  $T\otimes I_{X(\ell,1)}V_1^{-1}\bar{\X}\subseteq\X_1\otimes (\C\xi\oplus \mrx_2)$. \\
Therefore, 
$$\lambda_1(T)\bar{\X}=V_1T\otimes I_{X(\ell,1)}V_1^{-1}\bar{\X}\subseteq\bar{\X}$$
which shows that $\bar{\X}$ is an invariant subspace of $\lambda_1(T)$.  Since $T$ is arbitrary,  $\bar{\X}$ is an invariant subspace of $\lambda_1(\LL(\X_1))$. 
On the other hand,  the range of $ P_{\uplus,1}\lambda_2(\LL(\X_2))P_{\uplus,1}$ is contained in $\C\xi\oplus \mrx_2\subseteq\bar{\X}$. Let $P_{\bar{\X}}$ be the projection onto the subspace $\bar{\X}$.
 Now,  to compute mixed moments of random variables from $\lambda_1(\LL(\X_1))$ and $ P_{\uplus,1}\lambda_2(\LL(\X_2))P_{\uplus,1}$,  
 we just need to restrict everything on $\bar{\X}=P_{\bar{\X}}\X$.  
 By Proposition \ref{construction of monotone },  $\lambda_1(\LL(\X_1))$ is monotone to  $ P_{\uplus,1}\lambda_2(\LL(\X_2))P_{\uplus,1}$ with respect to $\phi_\xi$.
\end{proof}

The following proposition studies the relation between the free face of one pair and the Boolean face of another.
\begin{proposition}\normalfont 
 Let $\{(B_i,\beta_i),(C_i,\gamma_i)\}_{i\in I}$ be a family of   free-Boolean independent pairs  in $(\A,\phi)$ and  $L\subset I$. Then the subalgebra $\underset{i\in L}{\vee} B_i$ is monotone to  $\underset{i\in I\setminus L}{\vee} C_i$ in $(\A,\phi)$.
\end{proposition}
\begin{proof}
The statement is equivalent to showing that the algebra generated by $ \{ \lambda_i(\LL(X_i))\}_{i\in L}$ is monotone to the $ \{ P_{\uplus,i}\lambda_i(\LL(X_i))P_{\uplus,i}\}_{i\in I\setminus L}$ for arbitrary vector spaces with specified vectors $(\X_i,\mrx_i,\xi_i)_{i\in I}$ . Since the reduce free product is associative,  we have that
$$(\X, \xi)=\underset{i\in I}{*}(\X_i, \xi_i)=\Big(\underset{i\in L}{*}(\X_i,\xi_i)\Big)*\Big(\underset{i\in I\setminus L}{*}(\X_i, \xi_i)\Big).$$
Let $(\X',\xi')=\underset{i\in L}{*}(\X_i, \xi_i)$ and  $(\X'',\xi'')=\underset{i\in I\setminus L}{*}(\X_i, \xi_i)$.  For each $i$, let  $\lambda'_i: \LL(\X_i)\rightarrow \LL(\X')$ be the left regular representation of $ \LL(X_i)$ on  $ \X'$.  Let $\lambda': \LL(\X')\rightarrow \LL(\X)$ be the left regular representation of $ \LL(\X')$ on  $ \X$. Then,  we have $\lambda_i=\lambda'\circ\lambda'_i $. Therefore, $ \{ \lambda_i(\LL(X_i))\}_{i\in L}=\{ \lambda'\circ\lambda'_i (\LL(X_i))\}_{i\in L}\subset \lambda'(\LL(\X'))$. 
On the other hand,  let $P_{\uplus,''}$ be the projection onto the linear space $\C\xi\oplus(\bigoplus\limits_{i\in I\setminus L}\mrx_i)$ and vanishes on all the other tensor components.  Therefore, $P_{\uplus,''} P_{\uplus,i}= P_{\uplus,i}$ for all $i \in I\setminus L.$  Then, $\{ P_{\uplus,i}\lambda_i(\LL(X_i))P_{\uplus,i}\}_{i\in I\setminus L}\subset P_{\uplus,''} \lambda''(\LL(\X''))P_{\uplus,''}$ where $\lambda''$ is the left regular representation of $\LL(\X'')$ on $\X$.  
By Lemma \ref{FBtoM},   $\lambda'(\LL(\X'))$ is monotone to $P_{\uplus,''} \lambda''(\LL(\X''))P_{\uplus,''}$ with respect to $\phi_{\xi}$. The proposition follows.
\end{proof}

We next define the convolution associated with free-Boolean independence.

\begin{definition} \normalfont  Let $a=\{(b_i)_{i\in I},(c_j)_{j\in J}\}$ and $a'=\{(b'_i)_{i\in I},(c'_j)_{j\in J}\}$ be two pairs of two faced family of noncommutative random variables in a noncommutative  space  $(\A, \phi) $. 
We say that $a$ and $a'$ are \emph{free-Boolean independent} if $(B,C)$ and $(B',C')$ are free-Boolean independent, where $B,B'$ are the unital algebras generated by $(b_i)_{i\in I}$ and $(b'_i)_{i\in I}$ respectively  and $C,C'$ are the non-unital algebras generated by $(c_j)_{j\in J}$ and $(c'_j)_{j\in J}$ respectively. 
If $\{(b_i)_{i\in I},(c_j)_{j\in J}\}$ and $\{(b'_i)_{i\in I},(c'_j)_{j\in J}\}$ are free-Boolean independent, then the joint distribution of $\{(b_i+b'_i)_{i\in I},(c_j+c'_j)_{j\in J}\}$ is determined. This defines additive free-Boolean convolution $\boxplus\uplus$ on distributions of two-faced families of noncommutative random variables with pairs of index set $(I,J)$
$$\mu_{\{(b_i+b'_i)_{i\in I},(c_j+c'_j)_{j\in J}\}}=\mu_{\{(b_i)_{i\in I},(c_j)_{j\in J}\}}\boxplus\uplus \mu_{\{(b'_i)_{i\in I},(c'_j)_{j\in J}\} }.$$
The same we can define   multiplicative, additive-multiplicative, multiplicative-additive free-Boolean convolution. 
\end{definition}

\begin{remark} \normalfont 
In general, the sum of left faces of free-Boolean pairs is not commuting with the sum of right faces of free-Boolean pairs. Hence, free-Boolean convolution does not lead to a convolution of measures on  plane. 
\end{remark}

\section{Moment-conditions for free-Boolean independence}
In this section, we introduce a moments-condition definition for free-Boolean independence.

Let  $ (\mathcal{X}_i, \mrx_i,\xi_i)_{i\in I}$ be a family of  vector spaces with specified vectors and $( \mathcal{X}, \mrx,\xi)$ be  their  reduced free product.  
For each $i\in I$, let $\A_{i,\ell}=\lambda_i(\LL(\X_i))$ and $\A_{i,r}=P_{\uplus,i}\lambda(\LL(\X_i))P_{\uplus,i}$, where  $\lambda_i$ is the left regular representation and  $P_{\uplus,i}$ is the projection from $\X$ to $\C\xi\oplus\mrx_i$. According to Definition \ref{free-Boolean}, the family $\{\mathcal{A}_{i,\ell}, \mathcal{A}_{i,r}\}_{i \in I}$ is free-Boolean independent with respect to $\phi_\xi$, where $\phi_\xi$ is the linear functional associated with $\xi$ on $\LL(\X)$.

Let $ \A_i$ be the algebra generated by $\mathcal{A}_{i,\ell}$ and $\mathcal{A}_{i,r}$. 
Let $z_k\in \bigcup\limits_{i\in I } (A_{i,\ell}\cup A_{i,r})$ for $k=1,\cdots,n$.  Then the product  $z_1\cdots z_n$ can be  rewritten as 
$ Z_1\cdots Z_m$
for some $m$,
where $Z_k\in  \A_{i_k}$ and $i_{k-1}\neq i_k$.

For example, let $z_1\in \A_{1,\ell}$, $z_2\in \A_{1,r}$, $z_3\in \A_{1,\ell}$, $z_4\in \A_{2,\ell}$, and $z_5\in \A_{1,\ell}$.  Then 
$$z_1z_2z_3z_4z_5=(z_1z_2z_3)(z_4)(z_5)=Z_1Z_2Z_3,$$
where $Z_1=z_1z_2z_3\in\A_1$,  $Z_2=z_4\in\A_2$ and $Z_3=z_5 \in\A_1$.

\begin{definition}\normalfont 
Given  $z_1,\cdots, z_t\in \A_{i,\ell}\cup\A_{i,r} $ for some $i\in I$, their product $z_1\cdots z_t$ is called a \emph{Boolean product} if $z_k\in \A_{i,r}$ for some $k$.
\end{definition}

\begin{lemma}\normalfont  For each $i=1,2$, let $Z_i=\prod\limits_{k=1}^{s_i}z^{(i)}_{k,\beta_i(k)}$ where $\beta_i:[s_i]\rightarrow \{\ell, r\}$, $z^{(i)}_{k,\beta_i(k)}\in \A_{i,\beta_i(k)}$. If $Z_1$ is a Boolean product, then
$$ Z_1Z_2\xi=\phi_{\xi}(Z_2)Z_1\xi.$$
\end{lemma}
\begin{proof}
By direct computation,  $\mrx_{2}\oplus \C\xi$ is an invariant subspace of the $\A_{2,\ell}$ and $\A_{2,r}$. 
Therefore, $\mrx_{2}\oplus \C\xi$ is an invariant subspace of  $\A_{2}$. Then 
$Z_2\xi=\phi_{\xi}(Z_2)\xi+v$, where $v\in \mrx_{2}$. Since $Z_1$ is Boolean,  $\beta_1(k)=r $ for some $k$. Let $l$ be the largest number that $\beta_1(l)=r$. Then, the product $z^{(1)}_{l+1,\beta_{1}(l+1)}\cdots z^{(1)}_{s_1,\beta_1(n)}\in \A_{1,\ell}$.  According to the  definition of  $\A_{1,\ell}$, there exist an element $T\in \LL(\X_{1})$ such that $z^{(1)}_{l+1,\beta_{1}(l+1)}\cdots z^{(1)}_{s_1,\beta_1(n)}=\lambda_{1}(T)$.   Notice that 
$$
\begin{array}{rcl}
\lambda_{1}(T)v&=&V_1 T\otimes I_{\X_{1,\ell}}V_1^{-1}v\\
&=&V_1 (T\otimes I_{\X_{1,\ell})}(\xi_1\otimes v)\\
&=& V_1(T\xi_1\otimes v)\\
&=&\phi_{\xi_1}(T)v+(T-\phi_{\xi_1}(T)\xi_1)\otimes v,\\
\end{array}
$$  
which is a an element in $\mrx_2\oplus\mrx_1\otimes \mrx_2$.  However,  $\mrx_2\oplus\mrx_1\otimes \mrx_2$ is contained in the kernel of  $P_{\uplus, 1}.$  Therefore, $$Z_1v=[z^{(1)}_{1,\beta_{1}(1)}\cdots z^{(1)}_{l-1,\beta_1(l-1)}]z^{(1)}_{l,\beta_{1}(l)}[z^{(1)}_{l+1,\beta_{1}(l+1)}\cdots z^{(1)}_{s_1,\beta_1(n)}v]=0.$$
Thus  we have
$$Z_1Z_2\xi=Z_1(\phi_{\xi}(Z_2)\xi+v)=\phi_{\xi}(Z_2)Z_1\xi.$$
\end{proof}

The following proposition  follows immediately from the preceding lemma.
\begin{proposition}\label{Boolean Moments} \normalfont For each $i=1,\cdots,m$, let $Z_i=\prod\limits_{k=1}^{l_i}z^{(i)}_{k,\beta_i(k)}$ where $\beta_i:[l_i]\rightarrow \{\ell, r\}$, $z^{(i)}_{k,\beta_i(k)}\in \A_{j_i,\beta(k)}$. If all $Z_i$ are Boolean products and $j_i\neq j_{i+1}$ for all $k$, then
$$\phi_\xi(Z_1\cdots Z_m)=\phi_\xi(Z_1)\cdots \phi_\xi(Z_m).$$
\end{proposition}

\begin{proposition}\label{Free-boolean moments} \normalfont For each $i=1,\cdots,m$, let $Z_i=\prod\limits_{k=1}^{s_i}z^{(i)}_{k,\beta_i(k)}$ where $\beta_i:[s_i]\rightarrow \{\ell, r\}$, $z^{(i)}_{k,\beta_i(k)}\in \A_{j_i,\beta_i(k)}$. If  the following conditions hold:
\begin{itemize}
\item There exist $1\leq l_1< l_2\leq m$ such that $Z_{l_1}, Z_{l_1+1},\cdots, Z_{l_2}$ are not Boolean products.
\item $Z_{l_1-1}$ is a Boolean product or $l_1=1$,
\item $ Z_{l_2+1}$ is a Boolean product or $l_2=m$,
\item  $j_i\neq j_{i+1}$ for all $i$,
\item $\phi_\xi(Z_{l_1})=\phi_\xi(Z_{l_1+1})=\cdots=\phi_\xi(Z_{l_2})=0,$
\end{itemize}
then
$$\phi_\xi(Z_1\cdots Z_m)=0.$$
\end{proposition}
\begin{proof} Assume that those conditions are satisfied.
If  $l_2<m$,   then $Z_{l_2+1}\cdots Z_{m}\xi$ is   contained in $\mrx_{l_2+1}+\C \xi$ because $Z_{l_2+1}$ is a Boolean product.
It follows that $$Z_{l_1}\cdots Z_m\xi=Z_{l_1}\cdots Z_{l_2}(a\xi+v),$$
where $v\in \X_{l_2'}$ and $a\in \C$.   If $l_2=m$, then $v=0$, $a=1$ and $l'_2$ can be any index other than $l_2$. Since $Z_{l_1}, Z_{l_1+1},\cdots, Z_{l_2}$ are not Boolean products,  $Z_{i}\in\A_{j_{i},\ell}$ for $i=l_1,\cdots,l_2$. Notice that $ \phi_\xi(Z_{i})=0$ implies that $Z_i\xi\in \mrx_{j_i}$ for $i=l_1,\cdots,l_2$.
By  the definitions of $\{\A_{i,\ell}\}_{i\in I}$, we have 
$$Z_{l_1}\cdots Z_{l_2}(a\xi+v)=a(Z_{l_1}\xi)\otimes (Z_{l_1+1}\xi)\otimes\cdots\otimes (Z_{l_2}\xi)+(Z_{l_1}\xi)\otimes (Z_{l_1+1}\xi)\otimes\cdots\otimes (Z_{l_2}\xi)\otimes v,$$
which is a vector in $\mrx.$

If $l_1=1$, then $ Z_1\cdots Z_m\xi=Z_{l_1}\cdots Z_{l_2}(a\xi+v)\in \mrx$. Thus $\phi_\xi(Z_1\cdots Z_m)=0.$

If $l_1>1$,  then $Z_{l_1-1}$ a Boolean product. It follows that  $\beta_{l_1-1}(k)=r $ for some $k$. Let $l$ be the largest number that $\beta_{l_1-1}(l)=r$. 
Then, the product $z^{(l_1-1)}_{l+1,\beta_{l_1-1}(l+1)}\cdots z^{(l_1-1)}_{s_{l_1-1},\beta_{l_1-1}(s_{l_1-1})}\in \A_{j_{l_1-1},\ell}$. 
 According to the  definition of  $\A_{j_{l_1-1},\ell}$, there exist an element $T\in \LL(\X_{j_{l_1-1}})$ such that $z^{(l_1-1)}_{l+1,\beta_{l_1-1}(l+1)}\cdots z^{(l_1-1)}_{s_{l_1-1},\beta_{l_1-1}(s_{l_1-1})}=\lambda_{j_{l_1-1}}(T)$.  Let $$w=a(Z_{l_1}\xi)\otimes (Z_{l_1+1}\xi)\otimes\cdots\otimes (Z_{l_2}\xi)+(Z_{l_1}\xi)\otimes (Z_{l_1+1}\xi)\otimes\cdots\otimes (Z_{l_2}\xi)\otimes v.$$
  Then $w\in \X(j_{l_1-1},\ell)$ and $\phi_\xi(w)=0$.  Since $j_{l_1-1}\neq j_{l_1}$, we have 
$\lambda_{j_{l_1-1}}(T)=bw+v'\otimes w $ for some $b\in \C$ and $v'\in\X_{j_{l_1-1}}$. Since $bw+v'\otimes w$ s a vector the kernel of  $P_{\uplus, j_{l_1-1}}$, we get   $Z_{l_1-1}Z_{l_1\cdots}Z_{l_2}\xi=0.$ 
Thus  we have
$\phi_\xi(Z_1\cdots Z_m)=0.$

\end{proof}

Proposition \ref{Boolean Moments} and Proposition \ref{Free-boolean moments} provide us a recursive algorithm for computing mixed moments of Free-Boolean independent pairs of random variables in the following sense:
To compute the the mixed moments $\phi(z_1\cdots z_n)$ such that $z_k\in \bigcup\limits_{i\in I } (A_{i,\ell}\cup A_{i,r})$ for $k=1,\cdots,n$.  We  first turn  the product into the form of $Z_1\cdots Z_m$ such that $Z_i\in A_{j_i}$ and $j_k\neq j_{k+1}$.  If all $Z_i$ are Boolean products, then the mixed moments is $\phi_\xi(Z_1)\cdots \phi_\xi(Z_m)$ by Proposition \ref{Boolean Moments}. If there are non-Boolean products $Z_{l_1}, Z_{l_2},\cdots, Z_{l_s}$, then $Z_{l_k}=\phi(Z_{l_k})I_{\X}+Z^{\circ}_{l_k}$ where $Z^{\circ}_{l_k}=(Z_{l_k}-\phi(Z_{l_k})I_{\X})$. 
Then we have 
$$ \begin{array}{rcl}
\phi(Z_1\cdots Z_m)&=&\phi(\cdots(\phi(Z_{l_1})I_{\X}+Z^{\circ}_{l_1})\cdots (\phi(Z_{l_2})I_{\X}+Z^{\circ}_{l_2})\cdots)\\
&=&\phi(\cdots(Z^{\circ}_{l_1})\cdots (Z^{\circ}_{l_2})\cdots)+\text{rest.}\\
\end{array}
$$
By Proposition \ref{Free-boolean moments},  $\phi(\cdots(Z^{\circ}_{l_1})\cdots (Z^{\circ}_{l_2})\cdots)=0$.  On the other hand, all terms in rest will be rearranged in the form of $Z'_1\cdots Z'_t$ such that $t<m$.  

Therefore,  the mixed moments $\phi(z_1\cdots z_n)$ is uniquely determine by a linear combination of mixed moments in the form of $\phi(z'_1\cdots z'_{s})$ such that $s< n$. 
Thus we have the following  equivalent definition for free-Boolean independence under Moments conditions.
\begin{theorem}
 Let $((\mathcal{A}_{i,\ell}, \mathcal{A}_{i,r})_{i \in I}$ be a family of pairs of algebras in a non-commutative probability space $(\mathcal{A}, \varphi)$. 
Let $Z_k= z_{1,k}\cdots z_{s_k,k}$ such that $z_{i,k}\in  \A_{j_k,\ell}\cup \A_{j_k,r}$, $s_k>0$ and  $i_{k-1}\neq i_k$ for $k=1,...m$.  The family  $((\mathcal{A}_{i,\ell}, \mathcal{A}_{i,r})_{i \in I}$ is said to be free-Boolean if 
$$\phi(Z_1\cdots Z_m)=\phi(Z_1)\cdots \phi(Z_m)$$
whenever $Z_1,\cdots, Z_m$ are Boolean products, and 
$$\phi(Z_1\cdots Z_m)=0$$ whenever  there exist $1\leq l_1< l_2\leq m$ such that $Z_{l_1}, Z_{l_1+1},\cdots, Z_{l_2}$ are not Boolean products, 
 $Z_{l_1-1}$ is a Boolean product or $l_1=1$,
 $ Z_{l_2+1}$ is a Boolean product or $l_2=m$,
 and $\phi(Z_{l_1})=\phi_\xi(Z_{l_1+1})=\cdots=\phi(Z_{l_2})=0.$
\end{theorem}

\section{ Interval-noncrossing partitions}
 
 In this section, we  introduce the combinatorial tools  to characterize free-Boolean pairs of algebras. It is well know that free independence can be characterized by free cumulants  which are described by noncrossing partitions and Boolean independence can be characterized by partitions related to interval partitions.  To characterize free-Boolean pairs,  we will use a combination of noncrossing partitions and interval partitions.  
 
Here,  we start with some elementary combinatorial concepts.  Given a natural number $k$, we  denote by $[k]$ the set of $\{1,\cdots,k\}$.

\begin{definition}\label{partition} \normalfont  Let $S$ be an ordered set:
\begin{itemize}
\item[1.] A \emph{partition} $\pi$ of a set $S$ is a collection $\{V_1,\cdots,V_r\}$ of disjoint, nonempty sets whose union  is $S$.  The sets $V_1,\cdots,V_r$ are called the blocks of $\pi$. The collection of all partitions of $S$ will be denoted by $P(S)$.
\item[2.] Given two partitions $\pi$ and $\sigma$, we say $\pi\leq \sigma$ if each block of $\pi$ is contained in a block of $\sigma$. This relation is called the \emph{reversed refinement order}.
\item[3.]A partition $\pi\in P(S)$ is \emph{noncrossing} if there is no quadruple $(s_1,s_2,r_1,r_2)$ such that $s_1<r_1<s_2<r_2$, $s_1,s_2\in V$, $r_1,r_2\in W$ and $V,W$ are two different blocks of $\pi$.  
\item[4.]A partition $\pi\in P(S)$ is an \emph{interval} partition  if there is no triple $(s_1,s_2,r)$ such that $s_1<r<s_2$, $s_1,s_2\in V$, $r\in W$ and $V,W$ are two different blocks of $\pi$.  
\item[5.] A block $V$ of a partition  $\pi\in P(S)$ is said to be \emph{inner} if there is block $W\in \pi$ and  $s,t\in W$ such that $s<v<t$ for all $v\in V$. A block is \emph{outer} if it is not inner.
\item[6.] Let $\omega:[k]\rightarrow I$.  We denote by ker $\omega$ the element of $P([k])$ whose blocks are the sets $\omega^{-1}(i), i\in I$.  We denote by $s\sim t$ if $s,t$ are in a same block.
\end{itemize}
\end{definition}
It is obvious that interval partitions are noncrossing and every block of an interval partition is an outer block.  
 In the rest of this section,  we fix an natural number $n\in\mathbb{N}$.  
 Let  $\chi:[n]\rightarrow \{\bullet, \circ \}$ be a map  from the natural ordered set $\{1,\cdots,n\}$  to the set of colors $\{\bullet, \circ\}$.

\begin{definition}\normalfont  A partition $\pi\in \mathcal{P}(n)$ is said to be \emph{interval-noncrossing} with respect to $\chi$ if $\pi$  is  noncrossing  and  no element of $\chi^{-1}(\circ)$ is contained in an inner block of $\pi$.  We denote by $INC(\chi)$ the set of all interval-noncrossing partitions with respect to $\chi$.
\end{definition}

\begin{remark}\normalfont 
 Notice that $1$ and $n$ are never contained in any inner block,  $INC(\chi_1)=INC(\chi_2)$ if $\chi_1=\chi_2$ on $\{2,\cdots,n-1\}$.
\end{remark}

For example,  given two noncrossing partitions $\pi_1=\{\{1,3\},\{2\},\{4,5\}\}$ and $\pi_2=\{\{1\},\{2,3,5\},\{4\}\}$ of $\{1,2,3,4,5\}$.  Let $\chi^{-1}(\circ)=\{2,5\}$. Then,  $\pi_1\not\in INC(\chi)$ and $\pi_2\in INC(\chi)$.\\
\begin{picture}(120.00,42.00)(-30.00, 5.00)

\put(0.00,21.00){\line(0,1){6.00}}
\put(0.00,27.00){\line(1,0){12.00}}
\put(12.00,21.00){\line(0,1){6.00}}
\put(18.00,21.00){\line(0,1){6.00}}

\put(6.00,21.50){\line(0,1){3.50}}

\put(18.00,27.00){\line(1,0){6.00}}
\put(24.00,21.50){\line(0,1){5.50}}

\put(0.00,21.00){\circle*{1.00}}
\put(6.00,21.00){\circle{1.00}}
\put(12.00,21.00){\circle*{1.00}}
\put(18.00,21.00){\circle*{1.00}}
\put(24.00,21.00){\circle{1.00}}

\put(-0.90,17.00) {\footnotesize 1}
\put(05.10,17.00) {\footnotesize 2}
\put(11.10,17.00) {\footnotesize 3}
\put(17.10,17.00) {\footnotesize 4}
\put(23.10,17.00) {\footnotesize 5}
 

\put(50.00,21.00){\line(0,1){4.00}}
\put(56.00,27.00){\line(1,0){18.00}}
\put(62.00,21.00){\line(0,1){6.00}}
\put(68.00,21.00){\line(0,1){4.00}}

\put(56.00,21.50){\line(0,1){5.50}}

\put(68.00,27.00){\line(1,0){6.00}}
\put(74.00,21.50){\line(0,1){5.50}}

\put(50.00,21.00){\circle*{1.00}}
\put(56.00,21.00){\circle{1.00}}
\put(62.00,21.00){\circle*{1.00}}
\put(68.00,21.00){\circle*{1.00}}
\put(74.00,21.00){\circle{1.00}}

\put(49.0,17.00) {\footnotesize 1}
\put(55.10,17.00) {\footnotesize 2}
\put(61.10,17.00) {\footnotesize 3}
\put(67.10,17.00) {\footnotesize 4}
\put(73.10,17.00) {\footnotesize 5}

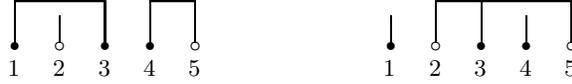
\put(0,5){{\it Figure 1.} Diagram of the partition $\pi_1$ and $\pi_2$}
\end{picture}

Furthermore, if $\chi^{-1}(\circ)=[n]$, then $INC(\chi)= I(n)$ where $I(n)$ is the set of interval partitions on $[n]$.  
If $\chi^{-1}(\circ)=\emptyset$, then $INC(\chi)= NC(n)$ where $NC(n)$ is the set of noncrossing partitions on $[n]$.

Now, we turn to study  relations between $INC(\chi)$ and noncrossing partitions. 
We will show that $INC(\chi)$ is a lattice for each coloring $\chi$. Since the values of $\chi$ at $1$ and $n$ do not change $INC(\chi)$, in the rest of this section, we will assume that $\chi(1)=\chi(n)=\circ$. Thus suppose that   $\chi^{-1}(\circ)=\{1=l_0<l_1<\cdots< l_m=n\}$. Given integers $n_1<n_2$, we denote by $[n_1,n_2]$ the interval $\{n_1,n_2+1,\cdots,n_2\}$. When $m\geq2$, we define the following maps :
\begin{itemize}
\item For each $i=1,\cdots,m$, let $\alpha_i: INC(\chi)\rightarrow P([l_{i-1},l_{i}])$ such that $\alpha_i(\pi)$ is the restriction of $\pi$ to $[l_{i-1},l_{i}]$.
\item Let $\alpha': INC(\chi)\rightarrow P([l_1,n])$ such that $\alpha'(\pi)$ is the restriction of $\pi$ to the  interval $[l_1,n]$.
\end{itemize}
Since restrictions of partitions on intervals just break some blocks of the original partitions, they do not turn any outer block into an inner block. Therefore, the range of $\alpha_1$ is a subset of $NC([1,l_1])$ and the range of $\alpha'$ is contained in  $INC(\chi')$  where $\chi'$ is the restriction of $\chi$ to the set $[l_1,\cdots,n]$. 

For example, let $n=10$, $\chi^{-1}(\circ)=\{1,3,7,8,9,10\}$ and $\pi=\{\{1,3,4,7\},\{2\},\{5,6\},\{9,8\},\{10\}\}$ which is interval-noncrossing with respect to $\chi$ as shown in the following diagram.\\

\begin{picture}(90.00,30.00)(-50.00, 5.00)

\put(0.00,21.50){\line(0,1){7.50}}
\put(12.00,21.50){\line(0,1){7.50}}
\put(18.00,21.00){\line(0,1){8.00}}
\put(0.00,29.00){\line(1,0){36.00}}
\put(36.00,21.50){\line(0,1){7.50}}

\put(6.00,21.00){\line(0,1){4.00}}

\put(24.00,21.00){\line(0,1){4.00}}
\put(24.00,25.00){\line(1,0){6.00}}
\put(30.00,21.00){\line(0,1){4.00}}

\put(48.00,21.50){\line(0,1){7.50}}
\put(42.00,29.00){\line(1,0){6.00}}
\put(42.00,21.50){\line(0,1){7.50}}

\put(54.00,21.50){\line(0,1){7.50}}

\put(0.00,21.00){\circle{1.00}}
\put(6.00,21.00){\circle*{1.00}}
\put(12.00,21.00){\circle{1.00}}
\put(18.00,21.00){\circle*{1.00}}
\put(24.00,21.00){\circle*{1.00}}
\put(30.00,21.00){\circle*{1.00}}
\put(36.00,21.00){\circle{1.00}}
\put(42.00,21.00){\circle{1.00}}
\put(48.00,21.00){\circle{1.00}}
\put(54.00,21.00){\circle{1.00}}

\put(-0.90,17.00) {\footnotesize 1}
\put(05.10,17.00) {\footnotesize 2}
\put(11.10,17.00) {\footnotesize 3}
\put(17.10,17.00) {\footnotesize 4}
\put(23.10,17.00) {\footnotesize 5}
\put(29.10,17.00) {\footnotesize 6}
\put(35.10,17.00) {\footnotesize 7}
\put(41.10,17.00) {\footnotesize 8}
\put(47.10,17.00) {\footnotesize 9}
\put(53.10,17.00) {\footnotesize 10}

\put(-5,9){{\it Figure 2.} Diagram of the partition $\pi$.}
\end{picture}

Then, $m=5$, $\alpha_1(\pi)=\{\{1,3\},\{2\}\}$,  $\alpha_2(\pi)=\{\{3,4,7\},\{5,6\}\}$, $\alpha_3(\pi)=\{\{7\},\{8\}\}$, $\alpha_4(\pi)=\{\{8,9\}\}$, $\alpha_5(\pi)=\{\{9\},\{10\}\}$  and $\alpha'(\pi)=\{\{3,4,7\},\{5,6\},\{8,9\},\{10\}\}$  are illustrated below:\\

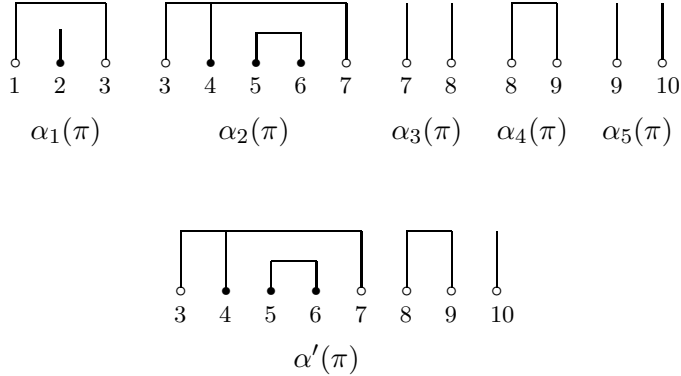
\begin{picture}(90.00,30.00)(-80.00, 5.00)

\put(-50.00,21.50){\line(0,1){7.50}}
\put(-38.00,21.50){\line(0,1){7.50}}
\put(-50.00,29.00){\line(1,0){12.00}}
\put(-44.00,21.50){\line(0,1){4.00}}

\put(-50.90,17.00) {\footnotesize 1}
\put(-44.90,17.00) {\footnotesize 2}
\put(-38.90,17.00) {\footnotesize 3}

\put(-50.00,21.00){\circle{1.00}}
\put(-44.00,21.00){\circle*{1.00}}
\put(-38.00,21.00){\circle{1.00}}

\put(-48,11){$\alpha_1(\pi)$}

\put(-30.00,21.50){\line(0,1){7.50}}
\put(-24.00,21.50){\line(0,1){7.50}}
\put(-30.00,29.00){\line(1,0){24.00}}
\put(-6.00,21.50){\line(0,1){7.50}}

\put(-18.00,21.00){\line(0,1){4.00}}
\put(-18.00,25.00){\line(1,0){6.00}}
\put(-12.00,21.00){\line(0,1){4.00}}

\put(-30.9,17.00) {\footnotesize 3}
\put(-24.9,17.00) {\footnotesize 4}
\put(-18.9,17.00) {\footnotesize 5}
\put(-12.9,17.00) {\footnotesize 6}
\put(-6.9,17.00) {\footnotesize 7}

\put(-30.00,21.00){\circle{1.00}}
\put(-24.00,21.00){\circle*{1.00}}
\put(-18.00,21.00){\circle*{1.00}}
\put(-12.00,21.00){\circle*{1.00}}
\put(-6.00,21.00){\circle{1.00}}

\put(-23,11){$\alpha_2(\pi)$}

\put(8.00,21.50){\line(0,1){7.50}}
\put(2.00,21.50){\line(0,1){7.50}}

\put(2.00,21.00){\circle{1.00}}
\put(8.00,21.00){\circle{1.00}}

\put(1.10,17.00) {\footnotesize 7}
\put(7.10,17.00) {\footnotesize 8}

\put(0,11){$\alpha_3(\pi)$}

\put(16.00,21.50){\line(0,1){7.50}}
\put(22.00,21.50){\line(0,1){7.50}}
\put(16.00,29.00){\line(1,0){6.00}}

\put(22.00,21.00){\circle{1.00}}
\put(16.00,21.00){\circle{1.00}}

\put(21.10,17.00) {\footnotesize 9}
\put(15.10,17.00) {\footnotesize 8}

\put(14,11){$\alpha_4(\pi)$}

\put(30.00,21.50){\line(0,1){7.50}}
\put(36.00,21.50){\line(0,1){7.50}}

\put(30.00,21.00){\circle{1.00}}
\put(36.00,21.00){\circle{1.00}}

\put(29.10,17.00) {\footnotesize 9}
\put(35.10,17.00) {\footnotesize 10}

\put(28,11){$\alpha_5(\pi)$}

\end{picture}

\begin{picture}(90.00,30.00)(-40.00, 5.00)

\put(12.00,21.50){\line(0,1){7.50}}
\put(18.00,21.50){\line(0,1){7.50}}
\put(12.00,29.00){\line(1,0){24.00}}
\put(36.00,21.50){\line(0,1){7.50}}

\put(24.00,21.50){\line(0,1){3.50}}
\put(24.00,25.00){\line(1,0){6.00}}
\put(30.00,21.50){\line(0,1){3.50}}

\put(48.00,21.50){\line(0,1){7.50}}
\put(42.00,29.00){\line(1,0){6.00}}
\put(42.00,21.50){\line(0,1){7.50}}

\put(54.00,21.50){\line(0,1){7.50}}

\put(12.00,21.00){\circle{1.00}}
\put(18.00,21.00){\circle*{1.00}}
\put(24.00,21.00){\circle*{1.00}}
\put(30.00,21.00){\circle*{1.00}}
\put(36.00,21.00){\circle{1.00}}
\put(42.00,21.00){\circle{1.00}}
\put(48.00,21.00){\circle{1.00}}
\put(54.00,21.00){\circle{1.00}}

\put(11.10,17.00) {\footnotesize 3}
\put(17.10,17.00) {\footnotesize 4}
\put(23.10,17.00) {\footnotesize 5}
\put(29.10,17.00) {\footnotesize 6}
\put(35.10,17.00) {\footnotesize 7}
\put(41.10,17.00) {\footnotesize 8}
\put(47.10,17.00) {\footnotesize 9}
\put(53.10,17.00) {\footnotesize 10}

\put(27,11){$\alpha'(\pi)$}
\end{picture}

Given  $\pi=\{V_1,\cdots,V_k \}\in INC(\chi)$. Suppose that $l_1\in V_j$ for some $j$. Let $W=\{t\leq l_1| t\in V_j\}$, $S=\{V_i| t<l_1, \forall t\in V_i\}$, $W'=\{t\geq l_1| t\in V_j\}$ and  $S'=\{V_i| t>l_1, \forall t\in V_i\}$.

\begin{lemma}\label{W1} \normalfont $\alpha_1(\pi)=S\cup \{W\}$ and $\alpha'(\pi)=S'\cup \{W'\}$.
\end{lemma}
\begin{proof}
It is sufficient to show that 
$$W\cup\bigcup\limits_{V_i\in S} V_i =\{1,\cdots,l_1\}.$$
Let $t\in V_i$ for some $i$,  $t<l_1$ and $l_1\not\in V_i$ . 
If there is a $t'\in V_i$ such that $t' > l_1$, then  according to the definition of interval-noncrossing partitions, $t, l_1, t'$ must be in the same block,  which contradicts our assumption. 
This shows that  all the elements of $V_i$ must be less than $l_1$, hence $V_i\in S$.  
Therefore,  for all $t<l_1$ and $t, l_1$ are not in the same block of $\pi$, $t$ must be contained in a block of $S$. 
On the other hand,  let  $t<\l_1$  such that $t, l_1$ are in the same block of $\pi$. Then  $t$ is  contained in $W$. 
Therefore,  $S\cup \{W\}$ is a partition of $\{1,\cdots,l_1\}$ and $\alpha_1(\pi)=S\cup \{W\}$. The same $\alpha'(\pi)=S'\cup \{W'\}$.
\end{proof}

\begin{lemma} \normalfont Let  $\alpha_1':INC(\chi)\rightarrow NC(l_1)\times INC(\chi')$ such that 
$$\alpha'_1(\pi)=(\alpha_1(\pi), \alpha'(\pi)). $$
Then $\alpha'_1$ is an order-preserving bijection from  $INC(\chi)$ to $NC(l_1)\times INC(\chi')$ where  the order on $NC(l_1)\times INC(\chi')$ is given by  $(\sigma_1,\sigma_1')\leq (\sigma_2,\sigma_2') $ if and only if $ \sigma_1\leq \sigma_2$ and  $ \sigma'_1\leq \sigma'_2$. Moreover, ${\alpha'_1}^{-1}$ is also order-preserving.
\end{lemma}
\begin{proof}First, we show that the map is well-defined and surjective.
Given two partitions $\sigma\in NC(l_1)$ and $\sigma'\in INC(\chi') $. 
Let $\sigma=S\cup \{W\}$ where $S$ is the family of blocks which does not contain $l_1$ and $W$ is the block which contains $l_1$, $\sigma'=S'\cup \{W'\}$ where $S'$ is the family of blocks which does not contain $l_1$ and $W'$ is the block which contains $l_1$,  $\pi=S\cup S' \cup \{W\cup W'\}$.  
Since $W$ and $W'$ contain the endpoints of the sets $\{1,\cdots,l_1\}$ and $\{l_1,\cdots, n\}$, they are outer blocks of $\sigma$ and $\sigma'$ respectively.
 Therefore, $W\cup W'$ is an outer block which contains $l_1$.  On the other hand,  $l_2,\cdots,l_m$ are contained in outer blocks of $S'$. Therefore, $\pi\in INC(\chi)$. By the construction of $\alpha'_1$, we have that $\alpha'_1(\pi)=(\sigma,\sigma'). $

Next, we show that $\alpha'_1$ is injectitive.
Given two partitions $\pi_1\neq\pi_2\in INC(\chi)$.  Suppose that $\alpha_1'(\pi_1)=\alpha_1'(\pi_2)=(S\cup \{W\},S'\cup \{W'\} )$ where $W$ and $W'$ are the blocks which contain $l_1$.  If $V$ is a block of $\pi_1$ which is contained in $S$ or $S'$, then  $V$ must be a block of $\pi_2$.  If $V$ is a block of $\pi_1$ which is not contained  in $S\cup S'$, then $V=W\cup W'$, by definition,  which is a block of $\pi_2$ also. Therefore, $\pi_1\leq \pi_2$. The same $\pi_2\leq \pi_1$, which implies $\pi_1=\pi_2$.

Now we show that $\alpha'_1$ is order-preserving.
 Let $\pi_1,\pi_2\in INC(\chi)$ such that $\pi_1\leq\pi_2$. Suppose that $\alpha_1'(\pi_1)=(S_1\cup \{W_1\},S_1'\cup \{W_1'\} )$ and $\alpha_1'(\pi_2)=(S_2\cup \{W_2\},S_2'\cup \{W_2'\} )$. Let $V$ be a block of  $S_1\cup \{W_1\}$. We need to show that $V$ is contained in a block of $S_2\cup \{W_2\}$. We distinguish two cases:\\
1.  If $V\in S_1$, then, by definition,  all the elements of $V$ are less than $l_1$. On the other hand,  $V$ is a subset of a block $V'$ of $\pi_2$.  If $l_1\not\in V'$, then $V'\in S_2$.   If $l_1\in V'$, then $V'=W_2\cup W_2'$, then $V\subseteq V'\cap \{1,\cdots,l_1\}=W_2$. Therefore, $V$ is contained in a block of $ S_2\cup \{W_2\}$. \\
2. If $V=W'_1$, then $W_1\cup W_1'$ is the block of $\pi_1$ which contains $l_1$. Because that $\pi_1\leq \pi_2$, $W_1\cup W_1'$ is contained in  $W_2\cup W_2'$ which is the block of $\pi_2$ contains $l_1$.   Therefore, $V= W_1\cup W_1'\cap \{1,\cdots,l_1\}= W_2\cup W_2'\cap \{1,\cdots,l_1\}=W_2$.  Thus, $V$ is contained in a block of $ S_2\cup \{W_2\}$.\\ Since $V$ is arbitrary, $S_1\cup \{W_1\}\leq S_2\cup \{W_2\}$. 
Similarly,  we have  $S'_1\cup \{W'_1\}\leq S'_2\cup \{W'_2\}.$
\end{proof}

Applying  the preceding lemma finitely many times to interval noncrossing partitions, we obtain the following result. 
\begin{proposition}\label{lattice isomorphism} \normalfont 
Let  $\alpha:INC(\chi)\rightarrow NC([1,l_1])\times NC([l_1,l_2])\times\cdots \times NC([l_{m-1},n])$ such that 
$$\alpha(\pi)=(\alpha_1(\pi), \cdots,\alpha_m(\pi)). $$
Then $\alpha$ is an  isomorphism of partial ordered sets. 
\end{proposition}
It follows from the preceding proposition that the poset $INC(\chi)$ is isomorphic to the poset $NC([1,l_1])\times NC([l_1,l_2])\times\cdots \times NC([l_{m-1},n])$. Since the second partial ordered set is a lattice, we have the following result.
\begin{proposition} \normalfont  $INC(\chi)$ is a lattice with respect to the reverse refinement order $\leq$ on partitions.
\end{proposition}

\begin{proposition}\normalfont   Let $\pi=\{V_1,\cdots,V_t\}\in INC(\chi)$ and let $\sigma$  be a partition of $[n]$ such that $\sigma\leq \pi$, i.e., each block of $\sigma$ is contained in a block of $\pi$. Then, $\sigma\in INC(\chi)$ if and only if $\sigma|V_s\in INC(\chi|_{V_s})$ for all $s=1,\cdots,t$.
\end{proposition}
\begin{proof}
Suppose  that $\sigma\in INC(\chi)$ and $s\in\{1,\cdots, t\}$.  
Then,  all the blocks of $\sigma|_{V_s}$ are blocks of $\sigma$ since $\sigma\leq \pi$.  
Let $W\in \sigma|_{V_s}$. If $l_k\in V_s$, $n_1,n_2\in W$ and $n_1<l_k<n_2$ for some $k$,  then $l_k\in W$ because $W$ is a block of an interval noncrossing partition with respect to $\chi$. Since $W$ and $l_k$ are arbitrary, $\sigma|_{V_s}\in INC(\chi|_{V_s})$

Conversely, suppose that $\sigma|V_s\in INC(\chi|V_s)$ for all $s=1,\cdots,t$.  Let  $W\in \sigma$, $n_1,n_2\in W$ and $n_1<l_k<n_2$ for some $k$. Since $\sigma\leq \pi$, $W\subseteq V_s$ for some $s$.  Then $n_1,n_2\in V_s$ and $l_k\in V_s$, because $\pi\in INC(\chi)$.  Note that $\sigma|V_s\in INC(\chi|V_s)$ and $l_k\in \chi^{-1}(\circ)\cap V_s$. It follows that we $l_k\in W$. Since $W$ and $l_k$ are arbitrary, $\sigma\in INC(\chi)$. 

\end{proof}

We see that the interval-noncrossing partitions, that are finer than a given partition $\pi$, are uniquely determined by the INC-partitions of the restrictions of $\chi$ to the blocks of $\pi$. This proves the following result.

\begin{corollary}\label{canonical isomorphism}\normalfont 
 Let $\pi=\{V_1,\cdots,V_t\}\in INC(\chi)$.  Then,  in $INC(\chi)$,
 $$[0_n, \pi]\cong INC(\chi|_{V_1})\times\cdots\times INC(\chi|{V_s}),$$
 where  $0_n$ is the partition of $[n]$ into $n$ blocks, and $[0_n,\pi]$ denote the interval $\{\sigma\in INC(\chi): 0_n\leq \sigma\leq \pi\}.$
\end{corollary}

\section{M\"obius functions on interval-noncrossing partitions} 
In free probability, the relation between the moments of variables and their cumulants are expressed by using  M\"obius functions on the lattice of noncrossing partitions. In this section,  we will develop  M\"obius functions on $INC(\chi)$. We first briefly review M\"obius functions on noncrossing partitions.

Let $L$ be a finite lattice. We denote by
$$L^{(2)}=\{(a,b)|b,a\in L,\,\, a\leq b\}$$
the set of order pairs of elements in $L$.\\
Given two functions $f,g:L^{(2)}\rightarrow \C$,  their convolution $f*g$ is given by:
$$f*g(a,b)=\sum\limits_{ \substack{c\in L\\ a\leq c\leq b}} f(a,c)g(c,b).$$
There are three special functions on $L^{(2)}$:
\begin{itemize}
\item The delta function defined as 
$$\delta(a,b)=\left\{\begin{array}{ll}
1,\,\,\,\, &\text{if}\,\, a=b,\\
0,&\text{otherwise.}
\end{array}\right.
$$
\item The \emph{zeta} function $\zeta$  defined as
$$\zeta(a,b)=\left\{\begin{array}{ll}
1,\,\,\,\, &\text{if}\,\, a\leq b,\\
0,&\text{otherwise.}
\end{array}\right.
$$
\item By proposition in \cite{Rota}, there is a  function $\mu$ on $L^{(2)}$ such that 
$$\mu*\zeta=\zeta*\mu=\delta. $$
$\mu$ is called the \emph{M\"obius function} of $L$.
\end{itemize}

Here, $\delta$ is the unit  with respect to the convolution $*$, and $ \mu$ is the inverse of $\zeta$ with respect to $*$.  Given lattice $L_1,\cdots, L_m$, then their direct product $L=L_1\times\cdots\times L_m$ is also a lattice with respect to the order such that
$(a_1,...,a_m)\leq(b_1,...,b_m)$ if and only if $a_i\leq b_i$ for all $i$. 
 It is obvious that $L^{(2)}=L_1^{(2)}\times\cdots\times L_m^{(2)}$. For each $i$, let $f_i$ be a $\C$-valued function on $L_i^{(2)}$. Then their  product $f=\prod\limits_{i=1}^mf_i$ is a function on $L^{(2)}$ defined as follows:
$$f((a_1,...,a_m),(b_1,...,b_m))= \prod\limits_{i=1}^mf_i(a_i,b_i),$$
for all  $(a_1,...,a_m),(b_1,...,b_m)\in L^{(2)}$. The M\"obius inversion functions of the lattice of non-crossing partitions are studied in \cite{Sp3}. 

\begin{lemma}\label{product of Mobius function} \normalfont  Let $L_1,\cdots,L_m$ be finite lattices. For each $i$, let $\delta_i$, $\zeta_i$, $\mu_i$ be the delta function, the zeta function and the M\"obius function of $L_i$  respectively. Then $\bar\delta=\prod\limits_{i=1}^m\delta_i$, $\bar\zeta=\prod\limits_{i=1}^m\zeta_i$ and $\bar\mu=\prod\limits_{i=1}^m\mu_i$ are  the delta function, the zeta function and the M\"obius function of $L_1\times\cdots\times L_m$.
\end{lemma}
\begin{proof}
It is obvious that $\bar\delta$ and $\bar\zeta$ are the delta function and zeta function of $L_1\times\cdots\times L_m$. We verify that $\bar\mu$ is  the corresponding M\"obius function.  Let  $a_i,b_i\in L_i$ and $a_i\leq b_i$.  We have 
$$
\begin{array}{rcl}
\bar\mu*\bar\zeta((a_i)_i,(b_i)_i)&=&\sum\limits_{\substack{(c_i)_i\in L_1\times\cdots\times L_m\\(a_i)_i\leq(c_i)_i\leq(b_i)_i}}\bar\mu((a_i)_i,(c_i)_i) \bar\zeta((c_i)_i,(b_i)_i)\\
&=&\prod\limits_{i=1}^m\sum\limits_{\substack{c_i\in L_i\\a_i\leq c_i\leq b_i}}\bar\mu_i(a_i,c_i) \bar\zeta(c_i,b_i)\\
&=&\prod\limits_{i=1}^m\sum\limits_{\substack{c_i\in L_i\\a_i\leq c_i\leq b_i}}\delta_i(a_i,b_i)\\
&=&\bar\delta((a_i)_i,(b_i)_i),\\
\end{array}
 $$
that is, $ \bar{\mu}*\bar{\zeta}=\bar\delta.$
\end{proof}

The following result follows in  immediately from Lemma \ref{product of Mobius function}.

\begin{proposition}\normalfont  Let $\alpha:N(\chi)\rightarrow NC([1,l_1])\times NC([l_1,l_2])\times\cdots \times NC([l_{m-1},n])$ be the poset isomorphism in Proposition \ref{lattice isomorphism}. 
Let $\bar\delta$, $\bar\zeta$, $\bar\mu$ be the delta function, the zeta function and the M\"obius function of $NC([1,l_1])\times NC([l_1,l_2])\times\cdots \times NC([l_{m-1},n])$.  Then $\delta_{INC(\chi)}=\bar\delta\circ\alpha$, $\zeta_{INC(\chi)}=\bar\zeta\circ\alpha$, $\mu_{INC(\chi)}=\bar\mu\circ\alpha$ are the delta function, the zeta function and the M\"obius function on ${INC(\chi)}$.
\end{proposition}

When there is no confusion, we simply write $\mu$ for the M\"obius inversion functions of $NC(n)$ for arbitrary $n$.  Let $(\sigma,\pi)\in INC(\chi)^{(2)}$. Let $\alpha(\sigma)=(\sigma_1,\cdots,\sigma_m),\alpha(\pi)=(\pi_1,\cdots,\pi_m)\in NC([1,l_1])\times NC([l_1,l_2])\times\cdots \times NC([l_{m-1},n])$.  Then, we have 
$$\mu_{INC(\chi)}(\sigma,\pi)=\prod\limits_{i=1}^m \mu(\sigma_i,\pi_i).$$
For convenience, we let $\mu({\emptyset},1_{\emptyset})=1$. Given a partition $\pi\in INC(\chi)$ and a blcok $V\in \pi$,  we set  $\tilde{\alpha}_i(V)=V\cap [l_{i-1},l_i]$, $i=1,\cdots, m$.  Since the M\"obius functions on $INC(\chi)$ depend only on $\chi$,  we  denote by $\mu_{INC}$ the M\"obius function of the lattice of interval-noncrossing partitions.

\begin{lemma}\label{Mobius transform1}\normalfont  Let $\pi=\{V_1,\cdots,V_t\}\in INC(\chi)$ and $\sigma\in INC(\chi)$ such that $\sigma\leq \pi$. Then,
$$\mu_{INC}(\sigma|_{V_s},1_{V_s})=\prod\limits_{i=1}^m \mu(\sigma_i|\talpha_i(V_s),1_{\talpha_i(V_s)}),$$
where $1_{\talpha_i(V_s)}$ is the partition of $\talpha_i(V_s)$ into one block.
\end{lemma}
\begin{proof}
Let $k_1=\min\{k|[l_k,l_{k+1}]\cap V_s\neq \emptyset\}$ and $k_2=\max\{k|[l_{k-1},l_{k}]\cap V_s\neq \emptyset\}$.  
Then, $l_k\in V_s$ for all $k_1<k<k_2$.  $V_s$ can be written as follows:
 $$V_s=\{n_1,n_2,\cdots,n_{i_1},l_{k_1},n_{i_1+1},\cdots,n_{i_2},l_{k_1+1},\cdots.,l_{k_2},n_{i_{k_2-k_1+1}+1},\cdots\}.$$
Therefore, $\chi^{-1}|_{V_s}(\circ)=\{l_k|k_1<k<k_2\}$.  Let $$\begin{array}{lcl}
W_1&=&\{n_1,n_2,\cdots,n_{i_1},l_{k_1}\}\\
W_2&=&\{l_{k_1},n_{i_1+1},\cdots,n_{i_2},l_{k_1+1}\}\\
&\cdots&\\
W_{k_2-k_1+2}&=&\{l_{k_2},n_{i_{k_2-k_1+1}+1},\cdots\},
\end{array}
$$
 be intervals of $V_s$ with respect to $\chi|_{V_s}$.  Then,  $W_j=\talpha_{k_1-2+j}(V_s)$ the restriction of $V_s$ to $[l_{k_1-2+j},l_{k_1-1+j}]$ for $j=1,\cdots,k_2-k_1+2$.

Since $\sigma\leq \pi$,  we have $\sigma|_{V_s}\in INC(\chi|_{V_s})$.  Notice that $W_j\subseteq V_s$,  the restriction of $\sigma|_{V_s}$ to $W_j$ is $\sigma|_{W_j}$.
Then,
$$\mu_{INC}(\sigma|_{V_s},1_{V_s})=\prod\limits_{j=1}^{k_2-k_1+2} \mu(\sigma|_{W_j},1_{W_j})=\prod\limits_{j=1}^{k_2-k_1+2} \mu(\sigma|_{W_j},1_{\talpha_{k_1-1+j}(V_s)}).$$

Since  $W_j\subseteq [l_{k_1-2+j},l_{k_1-1+j}]$, the restriction of $\alpha_{k_1-1+j}(\sigma)=\sigma|_{[l_{k_1-2+j},l_{k_1-1+j}]}$ to $ W_j$ is just $\sigma|_{W_j}$. In other words,  $ \sigma|_{W_j}=\alpha_{k_1-1+j}(\sigma)|_{W_j}=\alpha_{k_1-1+j}(\sigma)|_{\talpha_{k_1-1+j}(V_s)}$. Therefore,
$$
\begin{array}{rcl}
\mu_{INC}(\sigma|_{V_s},1_{V_s})&=&\prod\limits_{j=1}^{k_2-k_1+2} \mu(\sigma|_{W_j},1_{\talpha_{k_1-1+j}(V_s)})\\
&=&\prod\limits_{j=1}^{k_2-k_1+2} \mu(\alpha_{k_1-1+j}(\sigma)|_{\talpha_{k_1-1+j}(V_s)},1_{\talpha_{k_1-1+j}(V_s)})\\
&=&\prod\limits_{i=k_1}^{k_2+1} \mu(\sigma_i|\talpha_i(V_s),1_{\talpha_i(V_s)}).
\end{array}
$$
 For $i<k_1$ and $k>k_2+1$,  $\mu(\sigma_i|\talpha_i(V_s),1_{\talpha_i(V_s)})=1$ because $\talpha_i(V_s)=\emptyset$.  The lemma follows.
 \end{proof}

\begin{lemma}\label{Mobius transform}\normalfont  Let $\pi=\{V_1,\cdots,V_t\}\in INC(\chi)$ and $\sigma\in INC(\chi)$ such that $\sigma\leq \pi$. Then,
$$\mu_{INC}(\sigma,\pi)=\prod\limits_{s=1}^t\prod\limits_{i=1}^m \mu(\sigma_i|\talpha_i(V_s),1_{\talpha_i(V_s)}).$$
\end{lemma}
\begin{proof}
Notice that all blocks of $\pi_i$ come from the restriction of blocks of $\pi$ to the interval  $[l_{i-1},l_i]$ and   $\pi_i$, $\sigma_i$  are noncrossing partitions on $[l_{i-1},l_i]$. By Theorem 9.2 in \cite{NS},  the interval $[\sigma_i,\pi_i]$ of noncrossing partitions has the following  canonical factorization:
     $$[\sigma_i,\pi_i]\cong\prod\limits_{s=1}^t[\sigma_i|\talpha_i(V_s),1_{\talpha_i(V_s)}]$$
where $\sigma_i|\talpha_i(V_s)$ is the restriction of the partition $\sigma_i$ to the block $\talpha_i(V_s)$ and we allow the empty set in the above formula.    Then, by the multiplicative property of the M\"obius inversion on noncrossing partition intervals, we have 
$$ \mu(\sigma_i,\pi_i)=\prod\limits_{s=1}^{t} \mu(\sigma_i|\talpha_i(V_s),1_{\talpha_i(V_s)}),$$
and thus
$$
\begin{array}{rcl}
\mu_{INC(\chi)}(\sigma,\pi)&=&\prod\limits_{i=1}^m\prod\limits_{s=1}^{t} \mu(\sigma_i|\talpha_i(V_s),1_{\talpha_i(V_s)})\\
\end{array}
$$ 
\end{proof}
By  Lemma \ref{Mobius transform1} and \ref{Mobius transform}, we have the follow proposition.
\begin{proposition}\label{Mobius transform2} \normalfont
Let $\pi=\{V_1,\cdots,V_t\}\in INC(\chi)$ and $\sigma\in INC(\chi)$ such that $\sigma\leq \pi$. Then,
$$\mu_{INC}(\sigma,\pi)=\prod\limits_{s=1}^t\mu_{INC}(\sigma|_{V_s},1_{V_s}).$$
\end{proposition}

\section{ Vanishing cumulants condition for free-Boolean independence}
In this section, we introduce the notion of free-Boolean cumulants and  show that the  vanishing of joint free-Boolean cumulants is equivalent to free-Boolean independence.
\subsection{Free-Boolean cumulants}
Let $(\A,\phi)$ be a noncommutative probability space. For $n\in \mathbb{N}$, let $\phi^{(n)}$ be the \emph{$n$-linear} map from $\underbrace{\A\otimes\cdots\otimes \A}_{\text{n times}}$ to $\C$  defined as
$$\phi^{(n)}(z_1,\cdots,z_n)=\phi(z_1\cdots z_n),$$
where $z_1,...,z_n\in \A.$\\
Then, for each partition $\pi\in P(n)$,  we  define an $n$-linear map $\phi_{\pi}: \underbrace{\A\otimes\cdots\otimes \A}_{\text{n times}}\rightarrow \C$ recursively as follows:
$$\phi_\pi(z_1,\cdots,z_n)=\phi_{\pi\setminus V}(z_1,\cdots,z_l,z_{l+s+1},\cdots,z_n)\phi^{(s)}(z_{l+1},\cdots,z_{l+s}),$$ 
where $V=(l+1,l+2,\cdots,l+s)$ is an interval block of $\pi$.
\begin{definition}  \normalfont Given $\chi$ and $\pi\in INC(\chi)$, the \emph{free-Boolean cumulant} $\kappa_{\chi,\pi}$ is an $n$-linear map defined as follows:
$$\kappa_{\chi,\pi}(z_1,\cdots,z_n)
=\sum\limits_{\substack{\sigma\leq \pi\\ \sigma\in INC(\chi)}}\mu_{INC}(\sigma,\pi)\phi_{\sigma}(z_1,\cdots,z_n) .$$ 
\end{definition}

For example,  let $n=8$ and $\pi=\{\{1,5,8\},\{2,3,4\},\{6,7\}\}$. Then,
$$\phi_{\pi}(z_1,\cdots,z_8)=\phi_{\{\{1,5,8\},\{6,7\}\}}(z_1,z_5,z_6,z_7,z_8)\phi(z_2z_3z_4)=\phi(z_1z_5z_8)\phi(z_6z_7)\phi(z_2z_3z_4).$$

We show that free-Boolean cumulants have a multiplicative property.

\begin{theorem}\label{multiplicative property} \normalfont  Let $\pi=\{V_1,\cdots,V_t\}\in INC(\chi)$ and $z_1,\cdots,z_n$ be noncommutative random variables in a noncommutative probability space $(\A,\phi)$. Then
$$\kappa_{\chi,\pi}(z_1,\cdots,z_n)=\prod\limits_{s=1}^t \kappa_{\chi|_{V_s},1_{V_s}}(z_1,\cdots,z_n). $$
\end{theorem}
\begin{proof} By Lemma \ref{Mobius transform}, we have 
$$
\begin{array}{rcl}
&&\kappa_{\chi,\pi}(z_1,\cdots,z_n)\\
&=&\sum\limits_{\substack{\sigma\leq \pi\\ \sigma\in INC(\chi)}}\mu_{INC}(\sigma,\pi)\phi_{\sigma}(z_1,\cdots,z_n)\\
&=&\sum\limits_{\substack{\sigma\leq \pi\\ \sigma\in INC(\chi)}}[\prod\limits_{i=1}^m\prod\limits_{s=1}^{t} \mu(\sigma_i|\talpha_i({V_s}),1_{\talpha_i(V_s)})][\prod\limits_{s=1}^{t}\phi_{\sigma|_{V_s}}(z_1,\cdots,z_n)]\\
&=&\sum\limits_{\substack{\sigma\leq \pi\\ \sigma\in INC(\chi)}}\prod\limits_{s=1}^t[\prod\limits_{i=1}^{m} \mu(\sigma_i|\talpha_i(V_s),1_{\talpha_i(V_s)})\prod\limits_{s=1}^{t}\phi_{\sigma|_{V_s}}(z_1,\cdots,z_n)]\\
&=&\sum\limits_{\substack{\sigma\leq \pi\\ \sigma\in INC(\chi)}}\prod\limits_{s=1}^t[\prod\limits_{i=1}^{m} \mu(\sigma_i|\talpha_i(V_s),1_{\talpha_i(V_s)})\phi_{\sigma|_{V_s}}(z_1,\cdots,z_n)]\\
&=&\sum\limits_{\substack{\sigma\leq \pi\\ \sigma\in INC(\chi)}}\prod\limits_{s=1}^t[\mu_{INC}(\sigma|_{V_s},1_{V_s})\phi_{\sigma|_{V_s}}(z_1,\cdots,z_n)],\\
\end{array}
$$
where the last equality follows Lemma \ref{Mobius transform}. By corollary \ref{canonical isomorphism}, we have that 
$$
\begin{array}{rcl}
&&\sum\limits_{\substack{\sigma\leq \pi\\ \sigma\in INC(\chi)}}\prod\limits_{s=1}^t[\mu_{INC}(\sigma|_{V_s},1_{V_s})\phi_{\sigma|_{V_s}}(z_1,\cdots,z_n)]\\
&=&\prod\limits_{s=1}^t[\sum\limits_{\substack{ \sigma|_{V_s}\in INC(\chi|_{V_s})}}\mu_{INC}(\sigma|_{V_s},1_{V_s})\phi_{\sigma|_{V_s}}(z_1,\cdots,z_n)]\\
&=&\prod\limits_{s=1}^t \kappa_{\chi|_{V_s},1_{V_s}}(z_1,\cdots,z_n),
\end{array}
$$
thus the proof is complete.
\end{proof}

\begin{definition} \normalfont 
Let $\{(\mathcal{A}_{i,\ell}, \mathcal{A}_{i,r})\}_{i \in I}$ be a family of pairs of subalgebras of $\A$ in a non-commutative probability space $(\mathcal{A}, \varphi)$. We say that $\{(\mathcal{A}_{i,\ell}, \mathcal{A}_{i,r})\}_{i \in I}$ are \emph{combinatorially free-Boolean independent} if 
$$\kappa_{\chi,1_n}(z_1,\cdots, z_n)=0 $$
whenever  $\beta : \{1,\cdots,n\}\to \{\ell, r\}$, $\beta^{-1}(\ell)=\chi^{-1}(\bullet)$, $\omega : \{1,\cdots,n\}\to I$,  $z_k\in\A_{\omega(k),\beta(k)}$ and 
$\omega$ is not a constant.
\end{definition}

Notice that the condition $\beta^{-1}(\ell)=\chi^{-1}(\bullet)$ completely determines $\beta$.  In the follow context, we will always assume that $\beta^{-1}(\ell)=\chi^{-1}(\bullet)$. 
\begin{proposition}\normalfont 
Let $\{(\mathcal{A}_{i,\ell}, \mathcal{A}_{i,r})\}_{i \in I}$ be a family of pairs of algebras in a non-commutative probability space $(\mathcal{A}, \varphi)$. Then $\kappa_{\chi,1_n}$ has the following  cumulant property:
$$ \kappa_{\chi,1_n}(z_{1,1}+z_{2,1},\cdots, z_{1,n}+z_{2,n})=\kappa_{\chi,1_n}(z_{1,1},\cdots, z_{1,n})+\kappa_{\chi,1_n}(z_{2,1},\cdots, z_{2,n})$$
whenever $\omega_1,\omega_2:[n]\rightarrow I$,$z_{1,k}\in\A_{\omega_1(k),\beta(k)}$, $z_{2,k}\in\A_{\omega_2(k),\beta(k)}$  and $\omega([n])\cap\omega'([n])=\emptyset$.
\end{proposition}
\begin{proof}
By direct calculation, we have
$$\kappa_{\chi,1_n}(z_{1,1}+z_{2,1},\cdots, z_{1,n}+z_{2,n})=\sum\limits_{i_1,...i_n\in\{1,2\}}\kappa_{\chi,1_n}(z_{i_1,1},\cdots, z_{i_n,n}).$$
Since $\{(\mathcal{A}_{i,\ell}, \mathcal{A}_{i,r})\}_{i \in I}$ are combinatorially free-Boolean independent, by the preceding definition, we have 
$$\kappa_{\chi,1_n}(z_{i_1,1},\cdots, z_{i_n,n})=0$$
if $i_k\neq i_j$ for some $j,k\in[n]$. The result follows.
\end{proof}

\subsection{Free-Boolean is equivalent to combinatorially free-Boolean }In this subsection,  we will prove the following main theorem:

\begin{theorem}\label{Main} \normalfont 
Let $\{(\mathcal{A}_{i,\ell}, \mathcal{A}_{i,r})\}_{i \in I}$ be a family of pairs of faces in a non-commutative probability space $(\mathcal{A}, \varphi)$. $\{(\mathcal{A}_{i,\ell}, \mathcal{A}_{i,r}\}_{i \in I}$ are free-Boolean independent if and only if they are combinatorially free-Boolean.
\end{theorem}
It is sufficient to show  to  that mixed  moments  are uniquely determined by lower mixed moments in the same way for the two independence relations.  By Proposition 10.6 in \cite{NS} and Theorem \ref{multiplicative property}, we have  the following result.
\begin{lemma}\label{Moments-cumulant}  \normalfont  Let  $z_1,\cdots,z_n$ be noncommutative random variables in a noncommutative probability space $(\A,\phi)$.  Then
$$\phi(z_1\cdots z_n)=\sum\limits_{\pi\in INC(\chi)} \kappa_{\chi,\pi} (z_1,\cdots, z_n)$$
\end{lemma}

For combinatorially free-Boolean independent random variables, we have the following result.
\begin{lemma}\label{c-fb}\normalfont 
Let  $\{(\mathcal{A}_{i,\ell}, \mathcal{A}_{i,r})\}_{i \in I}$ be a family of  combinatorially free-Boolean independent  pairs of faces in  a noncommutative probability space $(\A,\phi)$.  Assume that $z_k\in\A_{\omega(k),\beta(k)}$, where $\omega:[n]\rightarrow I$, $\beta:[n]\rightarrow \{\ell, r\}$.  Let $\epsilon=\ker \omega$. Then, 
\begin{equation}\label{recursive relation}
\phi(z_1\cdots z_n)=\sum\limits_{\sigma\in INC(\chi)} \Big(\sum\limits_{\substack{ \pi\in INC(\chi)\\ \sigma\leq \pi\leq \epsilon}}\mu_{INC}(\sigma,\pi)\Big)\phi_{\sigma} (z_1,\cdots, z_n). \tag{$\bigstar$}
\end{equation}
\end{lemma}
\begin{proof}
By Lemma \ref{Moments-cumulant}, we have 
$$\phi(z_1\cdots z_n)=\sum\limits_{\pi\in INC(\chi)} \kappa_{\chi,\pi} (z_1,\cdots ,z_n).$$
For each $\pi\in INC(\chi)$, assume that $\pi=\{V_1,\cdots,V_t\}$. By Theorem \ref{multiplicative property}, we have  
$$\kappa_{\chi,\pi} (z_1\cdots z_n)=\prod\limits_{s=1}^t \kappa_{\chi|_{V_s},1_{V_s}}(z_1,\cdots,z_n). $$
Since $\{(\mathcal{A}_{i,\ell}, \mathcal{A}_{i,r})\}_{i \in I}$ are combinatorially free-Boolean,
$$\kappa_{\chi|_{V_s},1_{V_s}}(z_1,\cdots,z_n)=0$$
if $\omega$ is not a  constant on ${V_s}$.   It follows that $\kappa_{\chi,\pi}(z_1,\cdots,z_n)\neq 0$ only if $\omega$ is  a constant on $|_{V_s}$ for all $s$, which implies that $V_s$ is contained in a block of $\epsilon$ for all $s$, i.e.,  $\pi\leq \epsilon$.
 Therefore, we have
$$
\begin{array}{rcl}
\phi(z_1\cdots z_n)
&=&\sum\limits_{\pi\in INC(\chi),\pi\leq \epsilon} \kappa_{\chi,\pi} (z_1,\cdots, z_n)\\
                               &=&\sum\limits_{\pi\in INC(\chi),\pi\leq \epsilon} \Big(\sum\limits_{\substack{ \sigma\in INC(\chi)\\ \sigma\leq \pi}}\mu_{INC}(\sigma,\pi)\phi_{\sigma} (z_1,\cdots, z_n)\Big)\\
                               &=&\sum\limits_{\sigma\in INC(\chi)} \Big(\sum\limits_{\substack{ \pi\in INC(\chi)\\ \sigma\leq \pi\leq \epsilon}}\mu_{INC}(\sigma,\pi)\Big)\phi_{\sigma} (z_1,\cdots ,z_n).\\
\end{array}
$$
\end{proof}

Now, we turn to consider the case that the family $\{(\mathcal{A}_{i,\ell}, \mathcal{A}_{i,r})\}_{i \in I}$ is  free-Boolean independent in $(\A,\phi)$ in the sense of Definition \ref{free-Boolean}. 

We  assume that $z_k\in\A_{\omega(k),\beta(k)}$, where $\omega:[n]\rightarrow I$, $\beta:[n]\rightarrow \{\ell, r\}$. Let $\epsilon$ be the kernel of $\omega$. Let $\chi_1$ and $\epsilon_1$ be the restriction of $\chi$ and $\epsilon$ to the first interval $\{1,\cdots, l_1\}$ respectively.  
Let $\chi_1'$ and $\epsilon_1'$ be the restriction of $\chi$ and $\epsilon$ to the first interval $\{l_1,\cdots, n\}$ respectively.  We need to show that  the the mixed moments $\phi(z_1\cdots z_n)$ can be determined in the same way as in Lemma \ref{c-fb}. 

It is sufficient to consider the case that $\A=\LL(\X)$,  $\mathcal{A}_{i,\ell} =\lambda_i(\mathcal{L}(\X_i))$ and $\mathcal{A}_{i,r} =P_{\uplus,i}\lambda_i(\mathcal{L}(\X_i))P_{\uplus,i}$, where   $(\X_i, \mrx_i,\xi_i)_{i\in I}$ is a family of vector spaces with specified vectors and $(\mathcal{X}, \mrx,\xi)$ is the  reduced free product  of them.  $\phi=\phi_\xi$ is the linear functional associated with $\xi$ in $\X$.

We will prove the mixed moments formula $(\bigstar)$ in Lemma \ref{c-fb}   by  induction on the number of elements  of  $\chi^{-1}(\circ)\cap[2,n-1]$. 

\begin{lemma}\label{replace z_n} \normalfont  If $\chi(n)=\circ$, then there exists a $T\in \A_{\omega(n),\ell}$ such that 
$$\phi(z_1\cdots z_n)= \phi(z_1\cdots z_{n-1}T).$$
\end{lemma}
\begin{proof}
If $\chi(n)=\circ,$ then $z_n$ is from a Boolean face. 
Therefore, $z_n\in \A_{\omega(n),r}=P_{\uplus,\omega(n)}\lambda_{\omega(n)}(\mathcal{L}(\X_{\omega(n)}))P_{\uplus,\omega(n)}$. 
Assume that $z_n=P_{\uplus,\omega(n)}TP_{\uplus,\omega(n)}$ for some $T\in \lambda_{\omega(n)}(\mathcal{L}(\X_{\omega(n)}))$. Then 
$$z_1\cdots z_n\xi=z_1\cdots z_{n-1}P_{\uplus,\omega(n)}TP_{\uplus,\omega(n)}\xi=z_1\cdots z_{n-1}P_{\uplus,\omega(n)}T\xi=z_1\cdots z_{n-1}T\xi$$
since $\xi, T\xi\in P_{\uplus,\omega(n)}\X$.  Thus,  the mixed moments are the same if we replace $z_n$ by the element $T\in\lambda_{\omega(n)}(\mathcal{L}(\X_{\omega(n)}))$.
\end{proof}

\begin{lemma}\label{replace z_1} \normalfont If $\chi(1)=\circ$, then there exists a $T\in \A_{\omega(n),\ell}$ such that 
$$\phi(z_1\cdots z_n)= \phi(Tz_2\cdots z_{n}).$$
\end{lemma}
\begin{proof}
 If $\chi(1)=\circ,$ then $z_1\in \A_{\omega(1),r}=P_{\uplus,\omega(1)}\lambda_{\omega(1)}(\mathcal{L}(\X_{\omega(1)}))P_{\uplus,\omega(1)}$. Assume that $z_1=P_{\uplus,\omega(1)}TP_{\uplus,\omega(1)}$ for some $T\in \lambda_{\omega(1)}(\mathcal{L}(\X_{\omega(1)}))$.  Let $p$ be the projection that $p\xi=\xi$ and $p|\mrx=0$.  Then $pP_{\uplus,\omega(1)}=p$ and 
 
$$
\begin{array}{rcl}
\phi(z_1\cdots z_n\xi)
&=&\phi(P_{\uplus,\omega(1)}TP_{\uplus,\omega(1)}z_2\cdots z_{n}\xi)\\
&=&\phi(pP_{\uplus,\omega(1)}TP_{\uplus,\omega(1)}z_2\cdots z_{n}\xi)\\
&=&\phi(pTP_{\uplus,\omega(1)}z_2\cdots z_{n}\xi)\\
&=&\phi(TP_{\uplus,\omega(1)}z_2\cdots z_{n}\xi)\\
&=&\phi(Tz_2\cdots z_{n}\xi)-\phi(T(I_{\X}-P_{\uplus,\omega(1)})z_2\cdots z_{n}\xi),\\
\end{array}$$
where $I_{\X}$ is the unit of $\LL(\X)$. Notice that  $$(I_{\X}-P_{\uplus,\omega(1)})z_2\cdots z_{n}\xi\in \bigoplus\limits_{i\neq\omega(1)}\mrx_i\oplus\bigoplus\limits_{n\geq 2}\Big(\bigoplus\limits_{i_1\neq i_2\neq\cdots \neq i_n} \mathring{\mathcal{X}_{i_1}}\otimes\cdots \otimes\mathring{\mathcal{X}_{i_n}}\Big) $$
and 
$$\bigoplus\limits_{i\neq\omega(1)}\mrx_i\oplus\bigoplus\limits_{n\geq 2}\Big(\bigoplus\limits_{i_1\neq i_2\neq\cdots \neq i_n} \mathring{\mathcal{X}_{i_1}}\otimes\cdots \otimes\mathring{\mathcal{X}_{i_n}}\Big)=V_{\omega(1)}\big(\X_{\omega(1)}\otimes
\Big(\bigoplus\limits_{\omega(1)\neq i_1\neq i_2\neq\cdots \neq i_n} \mrx_{i_1}\otimes\cdots \otimes\mathring{\mathcal{X}_{i_n}}\Big) \big).$$
Therefore, $\bigoplus\limits_{i\neq\omega(1)}\mrx_i\oplus\bigoplus\limits_{n\geq 2}\Big(\bigoplus\limits_{i_1\neq i_2\neq\cdots \neq i_n} \mathring{\mathcal{X}_{i_1}}\otimes\cdots \otimes\mathring{\mathcal{X}_{i_n}}\Big)$ is an invariant subspace of $T$ and 
$$\phi(T(I_{\X}-P_{\uplus,\omega(1)})z_2\cdots z_{n}\xi)=\phi((I_{\X}-P_{\uplus,\omega(1)})T(I_{\X}-P_{\uplus,\omega(1)})z_2\cdots z_{n}\xi)=0. $$
The last equality follows that $p(I_{\X}-P_{\uplus,\omega(1)})=0.$ Therefore,  the mixed moments are the same if we replace $z_1$ by the element $T\in\lambda_{\omega(1)}(\mathcal{L}(\X_{\omega(1)}))$.

\end{proof}

When $|\chi^{-1}(\circ)\cap[2,n-1]|=0$, $\chi^{-1}(\circ)\cap[2,n-1]=\emptyset$. 
In this case,  $\chi$ can be $\circ$ only at $1$ and $n$.  Thus, $INC(\chi)=NC(n)$ which is set of noncrossing partitions on $[n]$.  Therefore, we have the following result.

\begin{lemma} \normalfont When $|\chi^{-1}(\circ)\cap[2,n-1]|=0$, we have 
$$\phi(z_1\cdots z_n)=\sum\limits_{\sigma\in INC(\chi)} \Big(\sum\limits_{\substack{ \pi\in INC(\chi)\\ \sigma\leq \pi\leq \epsilon}}\mu_{INC}(\sigma,\pi)\Big)\phi_{\sigma} (z_1\cdots z_n).\\$$
\end{lemma}
\begin{proof}
By Lemma \ref{replace z_n} and \ref{replace z_1}, we have 
$$ \phi(z_1\cdots z_n)=\phi(T_1z_2\cdots z_{n-1}T_2),$$
for $T_1\in \A_{\omega(1),\ell}$ and $ T_2\in\A_{\omega(n),\ell}$. Of course, if $z_1\in \A_{\omega(1),\ell}$, then the just let $T_1=z_1.$ The same to $z_n$.  Since $|\chi^{-1}(\circ)\cap[2,n-1]|=0$, $T_1, z_2,\cdots,z_{n-1}, T_2$ are from the left faces of algebras. Notice that 
$(\mathcal{A}_{i,\ell})_{i \in I}$ are  freely independent in $(\A,\phi)$, see \cite{NS}, we have $$
\begin{array}{rcl}
\phi(z_1\cdots z_n)&=&\phi(T_1z_2\cdots z_{n-1}T_2)\\
&=&\sum\limits_{\pi\in NC(n),\pi\leq \epsilon} \kappa_{\pi} (T_1,z_2,\cdots ,z_{n-1},T_2)\\
&=&\sum\limits_{\sigma\in NC(n)} \Big(\sum\limits_{\substack{ \pi\in NC(n)\\ \sigma\leq \pi\leq \epsilon}}\mu(\sigma,\pi)\Big)\phi_{\sigma} (T_1,z_2,\cdots ,z_{n-1},T_2)\\
                               
                               &=&\sum\limits_{\sigma\in INC(\chi)} \Big(\sum\limits_{\substack{ \pi\in INC(\chi)\\ \sigma\leq \pi\leq \epsilon}}\mu_{INC}(\sigma,\pi)\Big)\phi_{\sigma} (T_1,z_2,\cdots ,z_{n-1},T_2)
\end{array} 
$$
 The last equality follows from that $INC(\chi)=NC(n)$ when $|\chi^{-1}(\circ)\cap[2,n-1]|=0$.                             
According to  Lemma \ref{replace z_n} and \ref{replace z_1}, we get  $$\phi(z_1\cdots z_n)=\sum\limits_{\sigma\in INC(\chi)} \Big(\sum\limits_{\substack{ \pi\in INC(\chi)\\ \sigma\leq \pi\leq \epsilon}}\mu_{INC}(\sigma,\pi)\Big)\phi_{\sigma} (z_1,\cdots ,z_n).$$
\end{proof}
  
Now, we turn to prove our main theorem.
\begin{proof}[Proof of Theorem \ref{Main}]
Suppose that  Equation $(\bigstar)$ in Lemma \ref{c-fb} holds  whenever $|\chi^{-1}(\circ)\cap[2,n-1]|=m-2$.  For $|\chi^{-1}(\circ)\cap[2,n-1]|=m-1$,  we can assume $\chi^{-1}(\circ)=\{1=l_0<l_1<\cdots<l_{m-1}<l_m=n\}$.

Let $Z_1=\prod\limits_{i=l_1}^n z_i$.  Since  the range of $z_{l_1}$ is  $\C\xi\oplus \mrx_{\omega(l_1)}$,  $Z_1$ is a linear map from $\C\xi\oplus \mrx_{\omega(l_1)}$  to $\C\xi\oplus \mrx_{\omega(l_1)}$.  Moreover,  $Z_1$ vanishes on all summands of $\X$ except for $\C\xi\oplus \mrx_{\omega(l_1)}$. Therefore,  $Z_1$ is an element in $\lambda_{\omega(l_1)}(\LL(\X_{\omega(l_1)}))$. \\
By induction, we have 
$$
\begin{array}{rcl}
\phi(z_1\cdots z_n)&=&\phi(z_1\cdots z_{l_1-1}Z_1)\\
                               &=&\sum\limits_{\sigma_1\in NC(l_1),\sigma_1\leq \epsilon_1} \Big(\sum\limits_{\substack{ \sigma_1\in NC(l_1)\\ \sigma_1\leq \pi_1\leq \epsilon_1}}\mu_{}(\sigma_1,\pi_1)\Big)\phi_{\sigma_1}(z_1,\cdots ,z_{l_1-1}Z_{1}).\\
\end{array}
$$
Fix $\sigma_1$,  let  $V$ be the block of $\sigma_1$ which contains $l_1$. Then,
$$ 
\begin{array}{rcl}
\phi_{\sigma_1}(z_1\cdots z_{l_1-1}Z_{1})&=&\prod\limits_{W\in \sigma_1} \phi_{W}(z_1,\cdots, z_{l_1-1}Z_1)\\
&=&[\prod\limits_{W\neq V} \phi_{W}((z_1,\cdots, z_{l_1-1})]\phi_{V}(z_1,\cdots ,z_{l_1-1}Z_{1})\\
\end{array}
$$
the last equality holds since $l_1\not\in W$ whenever $W\neq V$.  Let $Z^V_{l_1}=\prod\limits_{i\in V}z_i$, where the product is taken with the original order. Then,
$$ \phi_{V}(z_1,\cdots ,z_{l_1-1}Z_{1})=\phi(Z_{l_1} ^V z_{l_1+1}z_{l_1+2}\cdots z_n).$$
Notice that $|\chi^{-1}[l_1+1,n-1]|=m-2$ . By assumption, we have
$$\phi(Z_{l_1} ^V z_{l_1+1}z_{l_1+2}\cdots z_n)=\sum\limits_{\sigma'\in INC(\chi')} \Big(\sum\limits_{\substack{ \pi'\in INC(\chi)\\ \sigma'\leq \pi'\leq \epsilon'}}\mu_{INC}(\sigma',\pi')\Big)\phi_{\sigma'} (Z_{l_1} ^V ,z_{l_1+1},z_{l_1+2},\cdots ,z_n). $$

Fix $\sigma'$,  assume that $l_1\in V'\in\sigma'$. Then
$$ 
\begin{array}{rcl}
\phi_{\sigma'} (Z_{l_1} ^V, z_{l_1+1},z_{l_1+2},\cdots, z_n)&=&\prod\limits_{W\in \sigma'} \phi_{W}(Z_{l_1} ^V ,z_{l_1+1},z_{l_1+2},\cdots ,z_n)\\
&=&[\prod\limits_{W\neq V'} \phi_{W}((z_{l_1+1},\cdots, z_{n})]\phi_{V'}((Z_{l_1} ^V ,z_{l_1+1},z_{l_1+2},\cdots,z_n)).\\
\end{array}
$$

Therefore,
$$\begin{array}{rcl}
&&\phi(z_1\cdots z_n)\\

&=&\sum\limits_{\substack{\sigma_1\in NC([l_1])\\\sigma_1\leq \epsilon_1}} \Big(\sum\limits_{\substack{ \sigma_1\in NC([l_1])\\ \sigma_1\leq \pi_1\leq \epsilon_1}}\mu_{}(\sigma_1,\pi_1)\Big)[\prod\limits_{\substack{W\in \sigma_1\\W\neq V\ni l_1}} \phi_{W}(z_1,\cdots, z_{l_1-1})]\phi_{V}(z_1,\cdots, z_{l_1-1},Z_{1})\\

&=&\sum\limits_{\substack{\sigma_1\in NC([l_1])\\\sigma_1\leq \epsilon_1}} \Big(\sum\limits_{\substack{ \sigma_1\in NC([l_1])\\ \sigma_1\leq \pi_1\leq \epsilon_1}}\mu_{}(\sigma_1,\pi_1)\Big)\Big\{[\prod\limits_{\substack{W\in \sigma_1\\W\neq V\ni l_1}} \phi_{W}(z_1,\cdots, z_{l_1-1})]\\

&&\sum\limits_{\substack{\sigma'\in INC(\chi')\\\sigma'\leq \epsilon'}} \Big(\sum\limits_{\substack{ \pi'\in INC(\chi)\\ \sigma'\leq \pi'\leq \epsilon'}}\mu_{INC}(\sigma',\pi')\Big)\phi_{\sigma'} (Z_{l_1} ^V, z_{l_1+1},z_{l_1+2},\cdots ,z_n)\Big\} \\

&=&\sum\limits_{\substack{\sigma_1\in NC([l_1])\\\sigma_1\leq \epsilon_1}} \sum\limits_{\substack{\sigma'\in INC(\chi')\\\sigma'\leq \epsilon'}} \Big(\sum\limits_{\substack{ \sigma_1\in NC([l_1])\\ \sigma_1\leq \pi_1\leq \epsilon_1}}\mu(\sigma_1,\pi_1)\Big)\Big(\sum\limits_{\substack{ \pi'\in INC(\chi)\\ \sigma'\leq \pi'\leq \epsilon'}}\mu_{INC}(\sigma',\pi')\Big)\Big\{[\prod\limits_{\substack{W\in \sigma_1\\W\neq V\ni l_1}} \phi_{W}(z_1,\cdots, z_{l_1-1})]\\

&&\phi_{\sigma'} (Z_{l_1} ^V ,z_{l_1+1},z_{l_1+2},\cdots ,z_n)\Big\} \\
\end{array}
$$

For fixed $\sigma_1$ and $\sigma'$, $\sigma=\alpha'^{-1}(\sigma_1,\sigma')\in INC(\chi)$ and
$$\mu_{INC}(\sigma,\pi)=\mu_{}(\sigma_1,\pi_1)\mu_{INC}(\sigma',\pi').$$
Since $\sigma_1\leq \epsilon_1$ and  $\sigma'\leq \epsilon'$,   we have $\sigma\leq \epsilon$.  
By the definition of $\alpha'$, a block $U$ of $\sigma$ can be one of the following three cases:
\begin{enumerate}
\item A block of $\sigma_1$ which does not contain $l_1$. In this case $\phi_U(z_1,\cdots,z_n)=\phi_{U}(z_1,\cdots, z_{l_1-1})$.
\item A block of $\sigma'$  which does not contain $l_1$. In this case $\phi_U(z_1,\cdots,z_n)=\phi_{U}(z_{l_1+1},\cdots, z_{n})$.
\item The block which contains $l_1$. In the case $U=V\cup V'$.  Then, 
$$ \phi_{V'}(Z_{l_1} ^V, z_{l_1+1},z_{l_1+2}\cdots z_n)=\phi(Z^V_{l_1}\prod\limits_{i\in V',i\neq l_1}z_i)=\phi(\prod\limits_{i\in V'\cup V,i\neq l_1}z_i)=\phi_U(z_1,\cdots,z_n.)$$
\end{enumerate}
Therefore,  
$$[\prod\limits_{\substack{W\in \sigma_1\\W\neq V\ni l_1}} \phi_{W}(z_1,\cdots, z_{l_1-1})]\phi_{\sigma'} (Z_{l_1} ^V ,z_{l_1+1},z_{l_1+2},\cdots ,z_n)=\phi_{\sigma}(z_1,\cdots,z_n).$$
and
$$\phi(z_1\cdots z_n)=\sum\limits_{\sigma\in INC(\chi)} \Big(\sum\limits_{\substack{ \pi\in INC(\chi)\\ \sigma\leq \pi\leq \epsilon}}\mu_{INC}(\sigma,\pi)\Big)\phi_{\sigma} (z_1,\cdots ,z_n).$$
Therefore, the mixed moments of free-Boolean independent random variables  and the mixed moments of combinatorially free-Boolean independent random variables  are determined in the same way. Thus the main theorem follows.
\end{proof}

\section{Central limit laws} 
In this section, we study an algebraic free-Boolean central limit theorem which is an analogou of Voiculescu's algebraic bi-free central limit theorem in \cite{Voi1}.   

Let $z=((z_i)_{i\in I},(z_j)_{i\in J})$ be a two faced family of noncommutative random variables in $(\A,\phi)$.  Let $\beta:[n]\rightarrow I\cup J$.   We denote by $\chi_\beta$ the map from $[n]$ to $(\circ,\bullet)$ such that $\chi_\beta(k)=\circ$ if and only if $\beta(k)\in J$.

  Notice that $INC(\chi)= NC(n)$ when  $n=1,2$ and $\chi$ is map from $[n]$ to $\{\circ,\bullet\}$. Therefore, the  second free-Boolean cumulants are, as the free cumulants, variance or covariance of random variables:
\begin{lemma}  \normalfont 
Let   $z=((z_i)_{i\in I},(z_j)_{i\in j})$ be a two faced family of noncommutative random variables in $(\A,\phi)$. Let $\beta: [2]\rightarrow I\cup J$. Then
$$\kappa_{\chi_\beta,1_{[2]}}(z_{\beta(1)}z_{\beta(2)})= \phi(z_{\beta(1)}z_{\beta(2)}) -\phi(z_{\beta(1)})\phi(z_{\beta(2)})$$
\end{lemma}

\begin{definition}\normalfont   A two faced family of noncommutative random variables $z=((z_i)_{i\in I},(z_j)_{i\in J})$ has a \emph{free-Boolean} central limit distribution if ,  for all $n\neq 2$,
$$\kappa_{\chi_\beta,1_{[n]}}(z_{\beta(1)},\cdots, z_{\beta(n)})= 0,$$
where $\beta:[n]\rightarrow I\cup J$.
\end{definition}

The following are  examples of free-Boolean families and   free-Boolean central limit distributions:

Let $\hh$ be a complex Hilbert space with orthonormal basis $\{e_i\}_{i\in I}$ and let $\F(\hh)=\C\xi\oplus\bigoplus\limits_{n\geq 1}\hh^{\otimes n}$ be the full Fock space. 
 Let $\ell_i$ be the left creation operators on $\F(\hh)$ such that 	$\ell_i \xi=e_i$ and $\ell_i\zeta=e_i\otimes \zeta$ for all $\zeta\in \bigoplus\limits_{n\geq 1}\hh^{\otimes n}$. 
 Let $P_i$ be the orthogonal projection from $\F(\hh)$ onto $\C\xi\oplus \C e_i$. Then the  family of two faced families 
 $$\{(l_i,l_i^*)_i,(P_il_iP_i,P_il_i^*P_i)_i\}_{i\in I}$$ 
 are free-Boolean in the probability space  $(B(\F(\hh)),\omega_\xi)$, where $(B(\F(\hh))$ is the set of bounded operators on $\F(\hh)$ and $\omega_\xi=\langle\cdot\xi,\xi\rangle$ is the vacuum state on $(B(\F(\hh)) $. 
 Actually,  $P_i$ here play the role of $P_{\uplus,i}$ in the Definition \ref{free-Boolean} and the space $\F(\C e_i)$, which is the Fock space generated by the one dimensional Hilbert space $\C e_i $, plays the role of $\X_i$.  
 
 Suppose that $I$ has a disjoint partition that $I=\bigcup\limits_{k\in K} I_k$. For each $k$, let $\A_{k,l}$ be the unital $C^*$-algebra generated by $\{\ell_i|i\in I_k\}$ and $\A_{k,r}$ be the nonunital $C^*$-algebra generated by $\{P_k\ell_iP_k|i\in I_k\}$, where $P_k$ is the projection from $\F(\hh)$ onto the subspace generated by $\{\xi\} \cup \{e_i|i\in I_k\}$. Then the family of pairs $(\A_{k,l},\A_{k,r})$ are free-Boolean in  $(B(\F(\hh)),\omega_\xi)$.

 The following is an analogue of Theorem 7.4 in \cite{Voi1}:
\begin{theorem}\label{central limit}\normalfont 
There is exactly one free-Boolean central limit distribution $\Gamma_C:\C\langle Z_k| k\in I\cup J \rangle \rightarrow \C$ for a given matrix $C=(C_{k,l})$ with complex entries so that $ \Gamma_C(Z_kZ_l)=C_{k,l}.$

Let $h,h^*: I\cup J\rightarrow \hh$ be  maps into the Hilbert space $\hh$ in the preceding example. 
Let 
$$ z_i=\ell(h(i))+\ell^*(h^*(i))$$ for $i\in I$ and 
$$ z_j=P(\ell(h(j))+\ell^*(h^*(j)))P$$ for $j\in J$, where $P$ is the orthogonal projection from $\F(\hh)$ onto $\C\xi\oplus\hh$ and $\ell(h(i))$ is the creation operator on $\F(\hh)$ such that $\ell(h(i))\xi=h(i)$ and $\ell(h(i))\zeta=h(i)\otimes \zeta$ for all $\zeta\in \bigoplus\limits_{n\geq 1}\hh^{\otimes n}$.
Then $z=((z_i)_{i\in I},(z_j)_{i\in J})$ has a free-Boolean central limit distribution $\Gamma_C$ where $C_{k,l}=\langle h(l),h^*(k)\rangle$.
\end{theorem}
\begin{proof}
The proof is the same as Voiculescu's.  
 \end{proof}

We end this section with an algebraic free-Boolean central limit theorem in analogue of Voiculescu's Theorem 7.9 in \cite{Voi1}

\begin{theorem}
Let $z_n=((z_{n,i})_{i\in I},(z_{n,j})_{i\in J})$, $n\in \mathbb{N}$, be a free-Boolean sequence of families of random variables in $(\A,\phi)$, such that 
\begin{itemize}
\item[1.] $\phi(z_{n,k})=0$, for all $k\in I\cup J$ and $n\in\mathbb{N}$.
\item[2.] $\sup_{n\in\mathbb{N}}|\phi(z_{n,k_1}\cdots z_{n,k_m})|=D_{k_1,\cdots,k_m}<\infty$ for every $k_1,\cdots,k_m\in I\cup J.$
\item[3.] $\lim\limits_{N\rightarrow \infty} \frac{1}{N}\sum\limits_{n=1}^N\phi(z_{n,k}z_{n,l})=C_{kl}$ for every $k,l\in I\cup J$.
\end{itemize}
Let $S_N=((S_{N,i})_{i\in I},(S_{N,j})_{i\in J})$, where $S_{N,k}=\frac{1}{\sqrt{N}}\sum\limits_{n=1}^N z_{n,k}$ with $k\in I\cup J.$  Let $F_C$ be the free-Boolean central limit distribution in Theorem \ref{central limit} with $C=(C_{k,l})$. Then, we have
$$\lim\limits_{N\rightarrow\infty} \mu_{S_N}(P)=\Gamma_C(P)$$
for all $P\in \C\langle Z_k|k\in I\cup J\rangle$.
\end{theorem}
\begin{proof}
Let $\beta:[m]\in I\cup J$. From Remark 6.4 and the $n$-linearity of $\kappa_{\chi_\beta,1_{[m]}}$, we have
$$\kappa_{\chi_{\beta},1_{[m]}}(S_{N,\beta(1)},\cdots ,S_{N,\beta(m)})=\sum\limits_{n=1}^N {N^{-\frac{m}{2}}}\kappa_{\chi_{\beta},1_{[m]}}(z_{n,\beta(1)},\cdots ,z_{n\beta(m)}).$$
When $m=1$, we have $\kappa_{\chi_{\beta},1_{[1]}}(z_{n,\beta(1)})=\phi((z_{n,\beta(1)})=0.$\\
When $m=2$, we have $\sum\limits_{n=1}^N {N^{-1}}\kappa_{\chi_{\beta},1_{[2]}}(z_{n,\beta(1)} z_{n\beta(2)})=\sum\limits_{n=1}^N {N^{-1}}\phi(z_{n,\beta(1)} z_{n\beta(2)})=C_{\beta(1),\beta(2)}$.\\
When $m>2$,  since  $D_{k_1,\cdots,k_m}<\infty$ for all $k_1,\cdots,k_m$, $\sup\limits_{n}|\kappa_{\chi_{\beta},1_{[m]}}(z_{n,\beta(1)},\cdots, z_{n\beta(m)}) |<\infty.$ Therefore, 
$$\lim\limits_{N\rightarrow \infty}\kappa_{\chi_{\beta},1_{[m]}}(S_{N,\beta(1)},\cdots, S_{N,\beta(m)})=0.$$
The proof is complete,  because moments are polynomials of cumulants.
\end{proof}

\vspace{1cm}
{\bf Acknowledgement} The author would like to thank professor Hari Bercovici and Dr. Ping Zhong for their careful reading and suggestions.  

\vspace{1cm}

\bibliographystyle{plain}

\bibliography{references}

\noindent Department of Mathematics\\
Indiana University	\\
Bloomington, IN 47401, USA\\
E-MAIL: liuweih@indiana.edu \\

\end{document}